\documentclass[oneeqnum]{siamltex}
\usepackage{mathptmx} 
\usepackage{array}
\usepackage{graphicx}
\usepackage{amssymb}
\usepackage{epstopdf}
\usepackage{amsmath}
\usepackage{amsfonts}
\usepackage{multirow}
\usepackage{color}
\usepackage[normalem]{ulem}
\usepackage{verbatim}
\usepackage{mathdots}
\usepackage{float}

\newcommand*\samethanks[1][\value{footnote}]{\footnotemark[#1]}

\newtheorem{assumption}[theorem]{Assumption}

\newcommand{\bu}{\mathbf{u}}
\newcommand{\bp}{\mathbf{p}}

\newcommand{\dx}{\dot{x}}

\newcommand{\dq}{\dot{q}}

\newcommand{\Tt}{^{{\mbox{\tiny \bf \sf T}}}}

\newcommand{\J}{{\cal{J}}}
\newcommand{\K}{{\cal{K}}}

\newcommand{\N}{{\cal{N}}}

\newcommand{\bx}{{\mathbf{x}}}
\newcommand{\bq}{{\mathbf{q}}}

\newcommand{\tU}{{\tilde{U}}}
\newcommand{\tu}{{\tilde{u}}}
\newcommand{\tx}{{\tilde{x}}}

\newcommand{\tbu}{{\tilde{\bu}}}

\def\setZ{\mathbb{Z}}
\def\setU{\mathbb{U}}
\def\setP{\mathbb{P}}
\def\setR{\mathbb{R}}
\expandafter\def\expandafter\normalsize\expandafter{%
    \normalsize
    \setlength\abovedisplayskip{5pt}
    \setlength\belowdisplayskip{5pt}
    \setlength\abovedisplayshortskip{5pt}
    \setlength\belowdisplayshortskip{5pt}
}

\begin{document}
\title{Structure of solutions for continuous linear programs with constant coefficients}

\author{Evgeny Shindin\thanks{
Department of Statistics,
The University of Haifa,
Mount Carmel 31905, Israel; 
email  {\tt shindin@netvision.net.il, gweiss@stat.haifa.ac.il}.
Research supported in part by
Israel Science Foundation Grants 711/09 and 286/13.
}  \and
Gideon Weiss \samethanks
}

\maketitle

\begin{abstract}
We consider Continuous Linear Programs over a continuous finite time horizon $T$, with linear cost coefficient functions, linear right hand side functions, and a constant coefficient matrix, as well as their symmetric dual.
We search for optimal solutions in the space of measures or of functions of bounded variation.  
These models generalize the Separated Continuous Linear Programming models and their various duals, as formulated in the past by Anderson, by Pullan, and by Weiss.    In a recent paper we have shown that  under a Slater type condition, these problems possess optimal strongly dual solutions.   In this paper we give a detailed description of optimal solutions and define a combinatorial analog to basic solutions of standard LP.  We also show that feasibility implies existence of strongly dual optimal solutions without requiring the Slater condition.  We present several examples to illustrate the richness and complexity of these solutions.
\end{abstract}

\medskip


\pagestyle{myheadings}
\thispagestyle{plain}
\markboth{E. SHINDIN AND G. WEISS}{STRUCTURE OF SOLUTIONS FOR A CLASS OF CLP}

%\begin{keywords} Continuous linear programming \end{keywords} 

%%%%%%%%%%%%%%%%%%%%%%%%%%%%%%%%%%%%%%%%%
% Introduction
%%%%%%%%%%%%%%%%%%%%%%%%%%%%%%%%%%%%%%%%%
\pagenumbering{arabic}

\section{Introduction}
\label{sec.introduction}
This paper continues our research in a recent paper \cite{shindin-weiss:13} on continuous linear programs of the form
\begin{eqnarray}
\label{eqn.mpclp}
&\max & \int_{0-}^T (\gamma+ (T-t)c)\Tt dU(t)  \nonumber\\
\mbox{M-CLP}\,\quad & \mbox{s.t.} & \qquad A\, U(t)  \quad \le \beta + b t, 
\quad 0 \le t \le T, \\
 &&  U(t) \ge 0,\; U(t) \mbox{ non-decreasing and right continuous on [0,T].}\nonumber 
\end{eqnarray}
and their symmetric dual,
\begin{eqnarray}
\label{eqn.mdclp}
&\min & \int_{0-}^T (\beta+(T-t)b)\Tt dP(t)  \nonumber \\
\mbox{M-CLP$^*$}\quad & \mbox{s.t.} & \qquad  A\Tt P(t)  \quad \ge \gamma + c t,  
\quad 0 \le t  \le T, \\
 &&   P(t) \ge 0,\; P(t) \mbox{ non-decreasing and right continuous on [0,T].}\nonumber 
\end{eqnarray}
Here $A$ is a $K\times J$ constant matrix, $\beta,b,\gamma,c$ are constant vectors of corresponding dimensions and the integrals are Lebesgue-Stieltjes.  The unknowns are vectors of cumulative control functions  $U$ and $P$,  over the time horizon $[0,T]$.  By convention we take $U(0-)=0,\,P(0-)=0$.
It is convenient to think of the dual as running in reversed time, so that $P(T-t)$ corresponds to $U(t)$. We will denote by $x(t)$ and $q(t)$ the slacks in (\ref{eqn.mpclp}), (\ref{eqn.mdclp}), and refer to them as the primal and dual states.

These problems are special cases of continuous linear programs (CLP) formulated by Bellman in 1953 \cite{bellman:53}, and further discussed and investigated in \cite{anderson-nash:87,anderson-philpott:89,dantzig:51,grinold:70,levinson:66,pullan:93,pullan:95,pullan:96,pullan:97,pullan:00,weiss:08}. The mnemonic M-CLP is used to stress that we  look for solutions in the space of measures.

In \cite{shindin-weiss:13} it was shown that if the primal and dual satisfy a Slater type condition, then optimal solutions exist, and strong duality holds.  Furthermore, the optimal solutions can be chosen to be absolutely continuous on $(0,T)$ with jumps, i.e. impulse controls only at $0$ and $T$.  It was also pointed out that the M-CLP formulation generalizes the SCLP (separated continuous linear programs) of \cite{anderson:81,pullan:93,weiss:08}

In this paper we focus on a detailed description of the solutions of M-CLP and M-CLP$^*$, and we present detailed examples that illustrate the surprising richness and complexity of these solutions.

Our contributions in this paper are:
\begin{itemize}
\item
Identify an extreme point (vertex) of M-CLP with a finite sequence of bases of an
associated LP.
\item
Give a detailed description of the optimal solution as it is computed for a given base-sequence.
\item
Derive the validity region of a base-sequences as a convex polyhedral cone
of model parameters for which the base-sequence is optimal.
\item
Give a simple non-degeneracy criterion, and prove uniqueness of the solution when it holds.
\item
Show that strong duality holds for any feasible pair M-CLP/M-CLP$^*$, without requiring the Slater type condition of \cite{shindin-weiss:13}.
\item
Point out the relation of M-CLP to SCLP, and illustrate by examples how M-CLP can solve SCLP problems which are not strongly dual.
\item
Analyze in detail all one dimensional M-CLP problems, and illustrate a complete solution of a 2 dimensional M-CLP. 
\end{itemize}  
Further
research to develop a simplex-type algorithm for  M-CLP, based on our identification of vertices as base sequences,  is in progress.

%%%%%%%%%%%%%%%%%%%%%%%%%%%%%%%%%%%%%%%%%
% Background and motivation
%%%%%%%%%%%%%%%%%%%%%%%%%%%%%%%%%%%%%%%%%

\section{Preliminaries}
In this section we state some  definitions and 
 briefly summarize the results of  \cite{shindin-weiss:13} on M-CLP/M-CLP$^*$.

\subsubsection*{$\bullet$ Complementary slackness}\leavevmode
\begin{definition}
Solutions of  M-CLP and M-CLP$^*$ are said to be complementary slack if
\begin{equation}
\label{eqn.compslack}
 \int_0^T x(T-t)\Tt dP(t) =  \int_0^T q(T-t)\Tt dU(t) = 0.
\end{equation}
\end{definition}
Optimal solutions are always complementary slack, and any feasible complementary slack solutions are optimal.
\subsubsection*{$\bullet$ Feasibility and non-degeneracy}
Throughout the  paper we will make following assumption:
\begin{assumption}[{\bf Feasibility}]
\label{asm.feas}
Both M-CLP and M-CLP$^*$ are feasible for target time horizon $T$.
\end{assumption}
\begin{assumption}[{\bf Non-degeneracy I}]
\label{asm.nondegI}
The vector $b$ is in general position to
the matrix $\left[ A \; I  \right]$
(it is not a linear combination of any less than $K$ columns), and the vector 
$c$ is in general position to the matrix $\left[ A\Tt \; I \right]$.
\end{assumption}
\begin{assumption}[{\bf Non-degeneracy II}]
\label{asm.nondegII}
The vector $\beta$ is in general position to the matrix $\left[ A \; I  \right]$, 
and the vector $\gamma$ is in general position to the matrix $\left[ A\Tt \; I \right]$.
\end{assumption}
\subsubsection*{$\bullet$ Criterion for feasibility}\leavevmode
 \begin{theorem}[Theorem 3.2  in \cite{shindin-weiss:13}]
\label{the.feas}
M-CLP is feasible if and only if the following standard linear program Test-LP with unknown vectors $\bu, U$ is feasible. 
 \begin{eqnarray}
\label{eqn.ptestLP}
 &\max &\; z = (\gamma + c T)\Tt   \bu + \gamma\Tt U  \nonumber  \hspace{1.0in} \\
 \mbox{Test-LP} \hspace{0.2in}&\mbox{s.t.} &  A  \bu  \le \beta  \\
 &\mbox{} &  A  \bu + A U  \le \beta + bT \nonumber \\
&& \quad \bu,\, U  \ge 0  \nonumber 
\end{eqnarray}
\end{theorem}
\subsubsection*{$\bullet$ Slater type condition}\leavevmode
\begin{definition}
\label{asm.slater}
We say that M-CLP satisfies Slater type condition,  if (\ref{eqn.ptestLP}) has a feasible solution  $\bu, U'$  such that $\beta - A\bu \ge \alpha$ and $\beta + bT - A\bu - AU' \ge \alpha$ for some $\alpha>0$.
\end{definition} 
\subsubsection*{$\bullet$ Existence of strongly dual optimal solutions}\leavevmode
\begin{theorem}[Theorem 4.7 in \cite{shindin-weiss:13}]
\label{thm.strongduality}
Under Slater type condition \ref{asm.slater} both M-CLP and M-CLP$^*$ have optimal solutions, and there is no duality gap.
\end{theorem}
\subsubsection*{$\bullet$ Partial characterization of the solutions}\leavevmode
\begin{theorem}[Proposition 5.2 and Theorem 5.5. in \cite{shindin-weiss:13}]
\label{thm.solform}
If  M-CLP/M-CLP$^*$ have optimal solutions $U_O(t), P_O(t)$, then: 

\noindent
(i)  $c\Tt U_O(t)$ is concave piecewise linear on $(0, T)$ with a finite number of breakpoints. 

\noindent
(ii) There exists an optimal solution $U^*(t)$ which is continuous piecewise linear on $(0, T)$. 

\noindent
(iii) Under the non-degeneracy assumption  \ref{asm.nondegI}, every optimal solution is of this form, and furthermore, $U_O(t)-U_O(0)$ is unique over $(0,T)$.
\end{theorem}

%%%%%%%%%%%%%%%%%%%%%%%%%%%%%%%%%%%%%%%%%
% Detailed description of the solution
%%%%%%%%%%%%%%%%%%%%%%%%%%%%%%%%%%%%%%%%%
\section{Detailed description of the solution}

%%%%%%%%%%%%%%%%%%%%%%%%%%%%%%%%%%%%%%%%%
%  Base sequences
%%%%%%%%%%%%%%%%%%%%%%%%%%%%%%%%%%%%%%%%%

In this section we give a detailed description of the solution.  We start by discussing the proposed form of the solution, and define {\em base sequences} in Section \ref{sec.basesequenses}.  These base sequences play a role analogous to bases in standard linear programming.
We then formulate and prove in Section \ref{sec.structurethm} {\em the  structure theorem} (Theorem \ref{thm.structure}), stating  that a solution is optimal if and only if it is obtained from a base sequence.  Two useful corollaries on {\em boundary values} then follow in Section \ref{sec.boundary}.  In Section \ref{sec.uniqueness} we show that under non-degeneracy assumption the solution is unique. Finally in Section \ref{sec.validity} we define {\em validity regions} for base sequences and show that they are given  by a convex polyhedral cone of the parameters of the problem.

\subsection{Base sequences}
\label{sec.basesequenses}

We consider solutions of M-CLP/M-CLP$^*$ that consist of impulse controls $\bu^0=U(0),\,\bu^N=U(T)-U(T-),\,\bp^0=P(T)-P(T-),\,\bp^N=P(0)$ at $0$ and $T$,  piecewise constant control rates $u(t)=\frac{dU(t)}{dt},\,p(t)=\frac{dP(t)}{dt}$, and continuous piecewise linear states $x(t)=\beta+bt-U(t),\,q(t)=\gamma+ct-P(t)$ %(the slacks in (\ref{eqn.mpclp}), (\ref{eqn.mdclp}))
with possible discontinuities at $T$.  We refer to the set of values of controls and states at $t=0$ and $t=T$ as {\em the boundary values}.  
The time horizon $[0,T]$ is partitioned by $0=t_0 < t_1 < \ldots < t_N=T$ which are the breakpoints in the rates $u,p$  and in the slopes of $x,q$.  We denote the values of the states at the breakpoints by   $x^n = x(t_n),\, n=0,\ldots,N-1$,
and $q^n=q(T-t_n)$, $n=1,\ldots,N$.  Because there may be a discontinuity at $T$ we denote the values at $T$ itself by $\bx^N=x(T)$, $\bq^0=q(T)$, and let $x^N=x(T-)$, $q^0=q(T-)$ be the values of the limit as $t\nearrow T$. 
The constant slopes of the states and the constant values of the control rates for each interval are denoted $\dx^n=\frac{d x(t)}{dt},\,u^n=u(t),\, t_{n-1}<t<t_n$ and  $\dq^n=\frac{d q(t)}{dt},\,p^n=p(t),\, T-t_n<t<T-t_{n-1}$.  Recall that the dual is running in reverse time, though we have labelled the intervals 
$n=1,\ldots,N$ in the direction of the primal problem.

For solutions of this form the solution is fully described by giving the boundary values, the rates for each interval, and the lengths of the  intervals $\tau_m=t_m-t_{m-1}$.  We will now discuss equations to calculate these.

The rates $u^n,\dx^n,p^n,\dq^n$ are complementary slack basic solutions of the following pair of 
LP problems
\begin{equation}
\label{eqn.prates}
\begin{array}{rcl}
 & \max &  c\Tt u  \\
 &\mbox{s.t.}& A u +  \dx = b \\
\mbox{Rates-LP} \quad&& u_j \in \setZ \text{ for } j \notin \J\; u_j \in \setP \text{ for } j \in \J \\
&& \dx_k \in \setU \text{ for } k \notin \K \; \dx_k \in \setP \text{ for } k \in \K 
\end{array} 
\end{equation}
\begin{equation}
\label{eqn.drates}
\begin{array}{rcl}
 & \min &  b\Tt p  \\
&\mbox{s.t.}& A\Tt p -  \dq = c \\
\mbox{Rates-LP$^*$}\quad && p_k \in \setZ \text{ for } k \notin \K\; p_k \in \setP \text{ for } k \in \K \\
&& \dq_j \in \setU \text{ for } j \notin \J \; \dq_j \in \setP \text{ for } j \in \J
\end{array}
\end{equation}
where by $\setZ, \setP, \setU$ we denote the following sign restrictions: $\setZ = \{0\}$ is zero, $\setP = \setR_+$ is non-negative and $\setU = \setR$ is unrestricted. The sign restrictions, which are determined by the sets of indexes $\K,\J$, vary from interval to interval.

We say that a basis $B$ of Rates-LP and the corresponding complementary slack dual basis $B^*$ of Rates-LP$^*$ are admissible if $u,p \ge 0$.  We say that two bases $B_n,B_{n+1}$ are adjacent if we can go in a single pivot operation from one to the other.   The piecewise constant values of $u(t),\dx(t),p(t),\dq(t)$ in the solution are given by a sequence of bases $B_1,\ldots,B_N$ which are admissible and adjacent.  We let:  
$\K_n=\{k: \dx_k^n \ne 0\}$, $\J_n=\{j: \dq_j^n \ne 0\}$, so that the basic variables of the primal Rates-LP for the $n$th interval are $\dx_k,u_j$ where $k\in \K_n,\,j\in \overline{\J}_n$.  In the pivot  $B_n \to B_{n+1}$, $v_n$  leaves the basis and $w_n$ enters the basis.  The values of $\dx^n,u^n,\dq^n,p^n$ can then be calculated from (\ref{eqn.prates}, \ref{eqn.drates}) once $B_1,\ldots,B_N$ are given. 

Next we have {\em time interval equations}.   If $v_n=\dx_k$ then $x_k(t_n)=0$, and if 
$v_n=u_j$ then $q_j(T-t_n)=0$.  We then get the following $N-1$ equations for $\tau_1,\ldots,\tau_N$:
 \begin{eqnarray}
\label{eqn.lineq1}
&\displaystyle x_k(t_n)=x^0_k + \sum_{m=1}^n  \dx^m_k \tau_m = 0, &\mbox{ if } v_n = \dx_k, 
\nonumber \\
&\displaystyle q_j(T-t_n)=q^N_j + \sum_{m={n+1}}^N  \dq^m_j \tau_m = 0, &\mbox{ if } v_n = u_j.
\end{eqnarray}
An additional equation is:
\begin{equation}
\label{eqn.sumtau}
\sum_{n=1}^N \tau_n = T.
\end{equation}

It remains to find the boundary values.   We let $\K_0 = \{k: x_k^0 > 0 \}$, $\J_0 = \{j: \bu_j^0 = 0\}$,  and 
$\J_{N+1} = \{j: q_j^N > 0\}$,  $\K_{N+1} = \{k: \bp_k^N = 0 \}$.   We say that $B_1$ and $\K_0$ are compatible if  $\K_0\subseteq \K_1$, 
and similarly, $B_N$ and $\J_{N+1}$ are compatible if  $\J_{N+1} \subseteq \J_N$.
The solution needs to satisfy these compatibility conditions.

By the definition of $\K_0,\J_0,\K_{N+1},\J_{N+1}$ and by complementary slackness,
\begin{equation}
\label{eqn.compatible}
\begin{array}{lllll}
\bu^0_j=0, &  j\in \J_0  &                                 \qquad \qquad \bp^N_k=0,  &  k\in \K_{N+1},  \\
x^0_k=0, &   k \not\in \K_0, &                          \qquad \qquad   q^N_j = 0, &  j\not\in \J_{N+1},  \\
\bp^0_k=0,  &  k\in \K_{0}, &  			 \qquad \qquad \bu^N_j=0, & j\in \J_{N+1}, \\
\bq^0_j = 0, & j\not\in \J_0, &			 \qquad \qquad \bx^N_k= 0, & k\not\in \K_{N+1}.
\end{array}
\end{equation}
This determines the value $0$ for $2K + 2J$ of the boundary variables.  
We note however that in contrast to $|\K_n|+|\overline{\J}_n|=K$ and $|\overline{\K}_n|+|\J_n|=J$ for $B_1,\ldots,B_N$,
all that we can say for the boundary values is that   $|\K_0|+|\overline{\J_0}|=K+L_1,\; |\overline{\K_0}|+|\J_0|=J-L_1$
and similarly 
 $|\J_{N+1}|+|\overline{\K_{N+1}}|=J+L_2,\;|\overline{\J_{N+1}}|+|\K_{N+1}|=K-L_2$. One can see that by assumption \ref{asm.nondegII} we always have $L_1\ge0,\,L_2\ge0$.
 
The remaining $2K +2J$ boundary values need to satisfy two sets of equations. The {\em  first boundary equations} relate to the constraints at time $0$ for the primal and the dual problem:
 \begin{equation}
\label{eqn.boundary1}
\begin{array}{llll}
A \bu^0  &+ x^0 = \beta, \quad  & \quad   A\Tt \bp^N  &- q^N = \gamma 
\end{array}
\end{equation}
The  {\em second boundary equations} relate to the discontinuities of $x$ and $q$ at $T$:
\begin{equation}
\label{eqn.boundary2}
A \bu^N  +  \bx^N - x^N = 0, \qquad  \qquad   A\Tt \bp^0   -  \bq^0 + q^0= 0.
\end{equation}

We can replace $x^N,\,q^0$ in (\ref{eqn.boundary2}) by  $x^N=x^0+\sum_{n=1}^N \dx^n \tau_n$ and $q^0=q^N+\sum_{n=1}^N \dq^n \tau_n$.  In this form the second boundary equations (\ref{eqn.boundary2}) together with first boundary equations (\ref{eqn.boundary1}) and time interval equations (\ref{eqn.lineq1}), (\ref{eqn.sumtau}) provide a set of $2K + 2J + N$ equations determining the $2K + 2J$ boundary values and the $N$ time interval lengths. 

We refer to the sequence $B_1,\dots,$ $ B_N$ together with $\J_0, \K_0, \J_{N+1}, \K_{N+1}$, or equivalently to the index sets $\K_n,\,\J_n,\,n=0,\ldots,N+1$ as the {\em base sequence}.

%%%%%%%%%%%%%%%%%%%%%%%%%%%%%%%%%%
%  The structure theorem
%%%%%%%%%%%%%%%%%%%%%%%%%%%%%%%%%%

\subsection{The structure theorem}\leavevmode
\label{sec.structurethm}
\begin{theorem}[Structure Theorem] 
\label{thm.structure} 
(i) Consider the M-CLP, M-CLP$^*$ problems (\ref{eqn.mpclp}), (\ref{eqn.mdclp}) and assume Non-Degeneracy Assumption \ref{asm.nondegI}. 
Let $B_1,\ldots,B_N$ be admissible adjacent bases with rates $u^n,\dx^n,p^n,\dq^n,\,n=1,\ldots,N$, and let 
$\K_0, \J_{N+1}$ be compatible with $B_1,B_N$.  
 Let $\bu^0,x^0,\bp^0,\bq^0$, $\bp^N,q^N,\bu^N,\bx^N$ and $\tau$ be a solution of  (\ref{eqn.lineq1})--(\ref{eqn.boundary2}), and let $u(t),x(t),p(t),q(t)$ be constructed from these boundary values, time intervals, and rates.  If all the boundary values, all the intervals, all the values $x^n,q^n$  at the breakpoints,  and the limit values  $x^N=x(T-),\,q^0=q(T-)$ 
 are $\ge0$, then this is an optimal solution.\\
(ii) Conversely, if problems (\ref{eqn.mpclp}), (\ref{eqn.mdclp}) are feasible, then there exists an optimal solution given by a sequence of admissible adjacent bases $B_1,\ldots,B_N$ and compatible boundary sets $\K_0, \J_0,$ $\K_{N+1}, \J_{N+1}$.
\end{theorem}
\begin{proof}
(i) The proof  is very similar to the proof of part (i) of the structure theorem in \cite{weiss:08}.  
We have $\tau \ge 0$ and by (\ref{eqn.sumtau})  they add up to $T$. Hence $0 = t_0 < t_1 < \dots < t_N = T$ is a partition of $[0, T ]$, $u(t), \dx(t), p(t), \dq(t)$ are well defined (at all but the breakpoints) piecewise constant and $x(t), q(t)$ are well defined continuous piecewise linear on $[0,T]$ with a possible discontinuity at $T$.

To show optimality we need to show that  $\bu^0, u(t), \bu^N, x(t), \bp^N, p(t), \bp^0, q(t)$, satisfy the constraints of (\ref{eqn.mpclp}), (\ref{eqn.mdclp}) as equalities, they are non-negative, and they are complementary slack as in (\ref{eqn.compslack}).

The primal and dual constraints hold at $t = 0$ by (\ref{eqn.boundary1}), hold for all $0 < t < T$ by integrating both sides of the constraints which involve $\dx, \dq$ in (\ref{eqn.prates}, \ref{eqn.drates}) from $0$ to $t$ and hence hold also at $t=T$ by (\ref{eqn.boundary2}).

Since $B_1,\ldots,B_N$ are admissible $u(t), p(t) \ge 0$.  That $x(t),q(t) \ge 0$ follows from  $x^n,q^n\ge 0,\,n=0,\ldots,N$ and 
$\bx^N,\bq^0 \ge 0$.

Next we show complementary slackness at all but the breakpoints. This is where we need to use the Non-Degeneracy Assumption (\ref{asm.nondegI}), as a result of which the following strict complementary slackness holds for all $t$ except the breakpoints:
\begin{equation}
\label{eqn.compslack2}
\begin{array}{lllllll}
x_k(t) > 0     & \iff &  \dx_k(t) \neq  0       & \iff & p_k(T-t)=0, & k=1,\dots,K,  \\
q_j(T-t) > 0   & \iff &  \dq_j(T-t) \neq  0    & \iff & u_j(t)=0,     & j=1,\dots,J.  \\
\end{array}
\end{equation}
The proof of (\ref{eqn.compslack2}) is given in the proof of Theorem 3 in \cite{weiss:08}.
%To prove (\ref{eqn.compslack2}) we show first that for all $t$ except the breakpoints, $x_k(t) \neq 0 \implies \dx_k(t) \neq 0$. Assume that $x_k(t) \neq 0$ where $t_n < t <t_{n+1}$. If $x_k(t_n)=0$, then $x_k(t) \neq 0$ implies by the continuous linearity of $x_k$ that the slope $\dx_k(t) \neq 0$. If $x_k(t_n) \neq 0$ we proceed by induction on $n$. For $n = 0$, $x_k(0) \neq 0$ implies $x^0_k > 0$, therefore by the compatibility of $\K_0$ with $B_1$, $\dx_k \in B_1$, and hence, by non-degeneracy, $\dx_k(t) \neq 0$. For $n > 1$, if  $x_k(t_n) \neq 0$ then by continuity $x_k(s) \neq 0$ for some $t_{n-1} < s < t_n$, which by the induction hypothesis implies $\dx_k(s) \neq 0$, hence $\dx_k \in B_n$. Assume that $\dx_k \not\in B_{n+1}$ then $\dx_k$ leaves the basis in the pivot $B_n \to B_{n+1}$. But the corresponding equation (\ref{eqn.lineq1}) for $n$ says in that case that $x_k(t_n)=0$ which is a contradiction. Hence, if $x_k(t_n) \neq 0$ then $\dx_k \in B_{n+1}$, and so $\dx_k(t) \neq 0$ by non-degeneracy. The proof for $q$ is the same. 
%
%The rest of the proof of (\ref{eqn.compslack2}) is immediate: The Non-Degeneracy Assumption (\ref{asm.nondegI}) implies $\dx_k(t) \neq 0 \iff p_k(T-t)=0$ and $\dq_j(T-t) \neq 0 \iff u_j(t)=0$ for all but the breakpoints. 
%For the last implication in (\ref{eqn.compslack2}), if $\dx_k^n\ne0$ then $x_k(t)$ is not constant, has a constant slope, and is non-negative in the interval $t_{n-1}<t<t_n$ which implies  $x_k(t)>0$.
% Similarly for $q$.

Finally,  complementary slackness holds at time $t=0$ and $t=T$ by (\ref{eqn.compatible}).

(ii) At this point we assume the Non-Degeneracy Assumption \ref{asm.nondegI} and Slater-type condition \ref{asm.slater}. We complete the proof, without these assumptions in Section \ref{sec.strong}. 
Let $U^*(t), P^*(t)$ be a pair of optimal solutions, as described in Theorem \ref{thm.solform}, with piecewise constant rates $u(t),p(t)$ and piecewise linear slacks $x(t), q(t)$, with breakpoints $0 =t_0 < t_1 < \dots < t_N = T$. We will construct an optimal base sequence from these solutions.

In each interval $u, \dx, p, \dq$ are optimal solutions of Rates-LP/LP$^*$ (\ref{eqn.prates}, \ref{eqn.drates}). By feasibility, 
complementary slackness and non-degeneracy, $u(t),\dx(t),p(t),\dq(t)$ must be  basic solutions of the Rates-LP/LP$^*$ problems (\ref{eqn.prates}, \ref{eqn.drates}), with admissible bases $B_1,\ldots,B_N$.    If these are adjacent, the proof is complete.  Else, if $B_n,B_{n+1}$ are not adjacent we can go from $B_n$ to $B_{n+1}$ in a sequence of pivots, preserving the admissibility.  In this way we will have a new sequence of bases $B_{1'},\ldots,B_{N'}$ which are feasible and adjacent.  The boundary values of 
$U^*(t), P^*(t)$ will then also determine $\K_0, \J_0, \K_{N+1}, \J_{N+1}$, where $\K_0, \J_{N+1}$ are compatible with $B_1, B_N$, because  $x(t), q(T-t)$ are right-continuous functions.
\end{proof}

It is seen from the structure theorem that the solution is determined by the base sequence $\K_n,\,\J_n,$ $n=0,\ldots,N+1$, and if the  conditions of the structure theorem are satisfied we call it {\em an optimal base sequence}.  We  refer to the constructed boundary values $\bu^0, \bu^N, x^0, \bx^N,$ $\bp^N, \bp^0, q^N, \bq^0$, the time partition $0=t_0 < t_1 < \dots < t_{N-1} < t_N = T$, and the control rates and states 
$u(t), x(t), p(t), q(t)$ as the optimal solution corresponding to the optimal base sequence.

Several Corollaries of Theorem 3 in \cite{weiss:08} are also valid here. In particular by Corollary 4  in \cite{weiss:08}, the objective values of Rates-LP/LP$^*$ are strictly decreasing.  This gives an upper bound on the number of intervals, $N \le \left(\begin{smallmatrix} K+ J \\ K \end{smallmatrix}\right)$. 
%Some important properties of these optimal solutions follow almost immediately from the Structure Theorem. 
%In the following Corollaries  we assume that $\K_n,\,\J_n,\,n=0,\ldots,N+1$  is an optimal base-sequence.

%%%%%%%%%%%%%%%%%%%%%%%%%%%%%%%%%
% Boundary values
%%%%%%%%%%%%%%%%%%%%%%%%%%%%%%%%%%

\subsection{Boundary values}
\label{sec.boundary}

The following corollaries determine some properties of\, the  boundary values, which are used later in the proof of uniqueness in Section \ref{sec.uniqueness}.
\begin{corollary}
\label{thm.boundaryLP}
Let $\check{U}=\int_0^T u(t) dt$, and $\check{P}=\int_0^T p(t) dt$.  Then $\bu^0,\bu^N$ and $\bp^0,\bp^N$ are optimal primal and dual solutions of the pair of dual LP problems:
\begin{equation}
\label{eqn.boundaryoptprim}
\mbox{Boundary-LP$(\check{U},\check{P},T)$ }\quad 
\begin{array}{ll}
\max & (\gamma +c T - A\Tt \check{P})\Tt \bu^0 + \gamma\Tt \bu^N   \\
\mbox{s.t.}& A \bu^0       \qquad \quad           \le \beta,  \\
                 & A \bu^0  + A \bu^N  \le \beta +  bT - A \check{U},  \\
& \bu^0,\,\bu^N \ge 0
\end{array}  \hspace{0.7in}
\end{equation}
\begin{equation}
\label{eqn.boundaryoptdual}
\mbox{Boundary-LP$^*(\check{U},\check{P},T)$}\quad 
\begin{array}{ll}
\max & (\beta +b T - A \check{U})\Tt \bp^N + \beta\Tt \bp^0   \\
\mbox{s.t.}&  A\Tt \bp^N       \qquad \quad           \ge \gamma,  \\
                 & A\Tt \bp^N  + A\Tt \bp^0  \ge \gamma +  c T - A\Tt \check{P}, \\
& \bp^N,\,\bp^0 \ge 0
\end{array}  \hspace{0.5in}
\end{equation}
\end{corollary}
\begin{proof}
Clearly, in the optimal solution $\bu^0,\bu^N,\bp^0,\bp^N$  are feasible solutions to (\ref{eqn.boundaryoptprim}), (\ref{eqn.boundaryoptdual}), with slacks $\bx^N,x^0$ for the primal and $\bq^0,q^N$ for the dual.  Furthermore, if $x^0_k > 0$ so that the corresponding constraint is not tight,  then by (\ref{eqn.compatible}) we have $\bp_k^0 = 0$, and vice versa, if $\bp_k^0 > 0$, then by (\ref{eqn.compatible}) $x_k^0 = 0$, and the corresponding constraint of the primal is tight.  The same argument works for the $\bx_k^N$, as well as  for the dual slacks.
Hence, $\bu^0,\bu^N,\bp^0,\bp^N$ are feasible complementary slack solutions, and so they are optimal.
\end{proof}

\begin{corollary}
\label{thm.boundaryLP1} 
Let the vectors $\underline{x}$ and $\underline{q}$ be defined by 
\[ \underline{x}_k = - \min_{1\le n\le N} \Big\{\sum_{m=1}^n\dx_k^m\tau_m ; 0 \Big\},  \quad
\underline{q}_j = - \min_{1\le n\le N} \Big\{\sum_{m=n}^N\dq_j^m\tau_m ; 0 \Big\}.
\]
Then the same boundary values $\bu^0,\bu^N$ and $\bp^0,\bp^N$ are also  optimal solutions of the following pair of (not dual)  LP problems:
\begin{equation}
\label{eqn.boundaryprob1}
\mbox{modBoundary-LP }\quad 
\begin{array}{ll}
\max & (\gamma +c T)\Tt \bu^0 + \gamma\Tt \bu^N   \\
\mbox{s.t.}& A \bu^0       \qquad \quad           \le \beta -\underline{x},  \\
                 & A \bu^0  + A \bu^N  \le \beta +  bT - A \check{U}, \\
& \bu^0,\,\bu^N \ge 0
\end{array}  \hspace{0.9in}
\end{equation}
\begin{equation}
\label{eqn.boundaryprob2}
\mbox{modBoundary-LP$^*$}\quad 
\begin{array}{ll}
\max & (\beta +b T)\Tt \bp^N + \beta\Tt \bp^0   \\
\mbox{s.t.}& A\Tt \bp^N       \qquad \quad           \ge \gamma + \underline{q},  \\
                 & A\Tt \bp^N  + A\Tt \bp^0  \ge \gamma +  c T - A\Tt \check{P}, \\
& \bp^N,\,\bp^0 \ge 0
\end{array}  \hspace{0.8in}
\end{equation}
\end{corollary}
\begin{proof}
Let $U,x,P,q$ be an optimal solution of M-CLP/M-CLP$^*$, as in the structure Theorem \ref{thm.structure}.
Substitute the optimal solution in the objective and constraints of M-CLP, to obtain:
\begin{equation}
\begin{array}{lll}
\label{eqn.dCLP1}
&& \displaystyle \mbox{Optimal objective}  =
(\gamma + c T)\Tt  \bu^0 + \gamma\Tt \bu^N + 
\sum_{m=1}^N \left(\gamma+ \left(T - \frac{t_m + t_{m-1}}{2}\right) c \right)\Tt \tau_m u^m
    \nonumber   \\
&&  A \bu^0 + x^0 = \beta   \nonumber \\[0.1cm]
&&  A \bu^0 + x^0  =  \beta +  x^0 + b t_n - A \sum_{m=1}^n \tau_m u^m - x^n   \quad \text{for } n=1,\dots,N   \\[0.1cm]
&&  A \bu^0 + A \bu^N  + \bx^N  = \beta + b T  - A \sum_{m=1}^N  \tau_m u^m    
\nonumber \\[0.1cm]
&& \quad \bu^0,\, u^1,\dots,u^N, \bu^N, \,x^0, x^1,\dots,x^N, \bx^N \ge 0.  \nonumber
\end{array}
\end{equation}
Since $\dx^n,u^n$ satisfy the constraints of the rates LP we have
\[
x^0 + b t_n - A \sum_{m=1}^n \tau_m u^m - x^n= \sum_{m=1}^n \tau_m (b - A u^m - \dx^m)  = 0.
\]
Also 
\[
x^0,x^1\ldots,x^n \ge 0   \mbox{ is equivalent to }  x^0 \ge \underline{x}.
\]
Keeping all the other optimal values, we see that $\bu^0,\, \bu^N, x^0, \, \bx^N$ are optimal solutions of
\begin{eqnarray}
\label{eqn.dCLP1m3}
 &\max &z =\displaystyle (\gamma + c T)\Tt  \bu^0 + \gamma\Tt \bu^N    \nonumber   \\
  &\mbox{s.t.} &  \displaystyle A \bu^0 + x^0 = \beta  \\
 &&  \displaystyle A \bu^0 + A \bu^N + \bx^N  = \beta + b T - A \check{U} \nonumber \\ 
 && x^0 \ge \underline{x}  \nonumber  \\
&& \quad \bu^0,\, \bu^N, \, \bx^N \ge 0.  \nonumber
\end{eqnarray}
which is equivalent to modBondary-LP (\ref{eqn.boundaryprob1}).   The proof for the dual boundary values is the same.
\end{proof}

%%%%%%%%%%%%%%%%%%%%%%%%%%%%%%%%%%%%%%%%%
% Uniqueness 
%%%%%%%%%%%%%%%%%%%%%%%%%%%%%%%%%%%%%%%%%
\subsection{Uniqueness}
\label{sec.uniqueness}
We now show that optimal solutions of the form described in the structure theorem are unique optimal solutions.  We need the following uniqueness condition:
\begin{definition}
\label{asm.uniqueness}{\bf Uniqueness Condition}
We say that $T$ satisfies the uniqueness condition if:
\[
\text{(1)  The vector }  \left[ \begin{array}{c}\beta \\ \beta + b T \end{array} \right] \text{ is in general position to the matrix }  \left[ \begin{array}{cccc} A & 0 & I & 0 \\ A & A & 0 & I \end{array} \right]
\]
\[
\text{(2)  The vector }   \left[ \begin{array}{c}\gamma \\ \gamma + c T \end{array} \right] \text{ is in general position to the matrix }  \left[ \begin{array}{cccc} A\Tt & 0 & I & 0 \\ A\Tt & A\Tt & 0 & I \end{array} \right]
\]
\end{definition}

{\sc Comment}  We believe that under  the non-degeneracy assumptions \ref{asm.nondegI}, \ref{asm.nondegII} the uniqueness condition is satisfied at all but a finite number of values of $T$.  We do not have a proof for this.

\begin{theorem}
\label{thm.uniqueness}
Let $U(t), x(t)$ be an optimal solution as described in the Structure Theorem \ref{thm.structure}. Assume the non-degeneracy assumption \ref{asm.nondegI}, and assume that $\beta, \gamma, T$ satisfy the Uniqueness Condition \ref{asm.uniqueness}.  Then $U(t), x(t)$ are the unique optimal solution to M-CLP.
\end{theorem}
\begin{proof}
By Theorem \ref{thm.solform} $U(t) - U(0)$ is unique on interval $(0, T)$. Hence, all we need to show is that the boundary values $\bu^0, \bu^N, x^0, \bx^N$ are unique. Assume otherwise. Let $\tU(t)$ be an optimal solution with another set of impulses $\tbu^0,\tbu^N$. By Corollary \ref{thm.boundaryLP1}, $\tbu^0,\tbu^N$ are optimal solutions of (\ref{eqn.boundaryprob1}), with exactly the same constants $\check{U}$ and $\underline{x}$.  
Consider then the dual of (\ref{eqn.boundaryprob1}), which has r.h.s $[\gamma, \gamma+ cT]\Tt$.   
 Under the Uniqueness Assumption \ref{asm.uniqueness}  all feasible basic solutions of this dual  are non-degenerate.  This implies that the optimal solution of (\ref{eqn.boundaryprob1}) is unique, and we have shown that  $\tbu^0=\bu^0,\,\tbu^N=\bu^N$.  As a result also $x^0,\,\bx^N$ are also unique.  
\end{proof}

%%%%%%%%%%%%%%%%%%%%%%%%%%%%%%%%%%
%Validity regions
%%%%%%%%%%%%%%%%%%%%%%%%%%%%%%%%%%

\subsection{Validity regions}
\label{sec.validity}\leavevmode
\begin{definition}
Let $\K_n,\,\J_n,\,n=0,\ldots,N+1$ be a base sequence. Let $\mathcal{T} $  be the set of all $\beta, \gamma, T$ for which this base sequence is optimal. Then  $\mathcal{T}$ is called the validity region of $\K_n,\,\J_n,\,n=0,\ldots,N+1$.
\end{definition}
\begin{theorem}
\label{thm.conves}
The validity region $\mathcal{T}$ of a base-sequence $\K_n,\,\J_n,\,n=0,\ldots,N+1$ is a convex polyhedral cone.
\end{theorem}
\begin{proof}
Let $\K_n,\,\J_n,\,n=0,\ldots,N+1$ be an optimal base sequence for at least one $(\beta, \gamma, T)$ and let
\[
v = \left[ \begin{array}{ccccccccccccccc}
\bu^0 & x^0 & q^N & \bp^N & \tau & x^1 & \dots & x^N & q^{N-1} & \dots & q^0 & \bu^N & \bx^N & \bp^0 & \bq^0 
\end{array}\right]
\]
be the corresponding optimal solution. 

For a given a column index set $\N$ we use the following notations: \\
 $A_\N$ is the matrix composed of the corresponding columns of the matrix $A$. \\
 $I_{[\N]}$ is the matrix obtained from the unit matrix $I$ by replacing the diagonal elements with indexes  in $\N$ by the value 0. 

We also define following  matrices:\\
$S$ is the $(N-1) \times N$ matrix of the $\dx, \dq$ coefficients of the time interval equations (\ref{eqn.lineq1}).\\
$E^1$ is the  $(N-1) \times K$ matrix composed of: $E^1_{n,k} = 1$ if for the time interval equations of $t_n$: $v_n = \dx_k$ and $x^0_k \in \K_0$, otherwise $E^1_{n,k} = 0$. \\ 
$E^2$ is the $(N-1) \times J$ matrix composed of:  $E^2_{n,j} = 1$ if for the time interval equations of $t_n$: $v_n = u_j$ and $q^N_j \in \J_{N+1}$, otherwise $E^2_{n,j} = 0$. \\
$E^3$ is the $(NK) \times K$ matrix where the left $((N-1)K) \times K$ block is a zero matrix and the right $K \times K$ block is a unit matrix. \\
$E^4$ is the $(NJ) \times J$ matrix where the left $((N-1)J) \times J$ block is a zero matrix and the right $J \times J$ block is a unit matrix. 
\[
X = \left[ \begin{array}{cccc}
\dx^1 & 0 &\dots & 0 \\
\vdots & \ddots &  & \vdots \\
\vdots & & \ddots & 0 \\
\dx^1 & \dx^2 &\dots & \dx^N 
\end{array}\right] \qquad
Q = \left[ \begin{array}{cccc}
0 & \dots & 0 & \dq^N \\
\vdots & & \iddots & \vdots \\
0 & \iddots &  & \vdots \\
\dq^1 & \dots & \dq^{N-1} & \dq^N 
\end{array}\right]
\]
Finally, we define the matrix $M$ \\
\scalebox{0.93}{
$
\left[
\begin{array}{ccccccccccc}
A_{\J_0} & I_{[\K_0]} & 0 & 0 &  0 & 0 & 0 & 0 & 0 & 0& 0  \\
 0 &  0 & - I_{[\J_{N+1}]} & A\Tt_{\K_{N+1}} & 0 & 0 & 0 & 0 & 0 & 0 & 0\\
  0 &  E^1 & E^2 & 0 & S  & 0 & 0 & 0 & 0 & 0 & 0\\
  0 &  0 & 0 & 0 & 1\dots 1 & 0 & 0 & 0 & 0  & 0 & 0\\
  0 &  I_{[\K_0]} & 0 & 0 & \multirow{3}{*}{$\Bigg[X\Bigg]$} &  \multirow{3}{*}{$\Bigg[- I\Bigg]$} & 0 & 0 & 0 & 0 & 0\\
  \vdots &  \vdots & \vdots & \vdots &  &   & \vdots & \vdots & \vdots & \vdots & \vdots\\
  0 &  I_{[\K_0]} & 0 & 0 &  &  & 0 & 0 & 0 & 0 & 0\\
  0 &  0 & I_{[\J_{N+1}]} & 0 & \multirow{3}{*}{$\Bigg[Q\Bigg]$} &  0 & \multirow{3}{*}{$\Bigg[- I\Bigg]$} & 0 & 0 & 0 & 0 \\
  \vdots &  \vdots & \vdots & \vdots &  & \vdots  &  & \vdots & \vdots & \vdots & \vdots\\
   0 &  0 & I_{[\J_{N+1}]} & 0 &  &  0 &  & 0 & 0 & 0 & 0 \\
  0 &  0 & 0 & 0 & 0  & -E^3 & 0 & A_{\J_{N+1}} & I_{[\K_{N+1}]} & 0 & 0 \\
  0 &  0 & 0 & 0 & 0 & 0 & E^4 & 0 & 0 & A\Tt_{\K_0} & - I_{[\J_0]} 
\end{array} \right]
$
}

Given  $A, b, c$, the  base sequence $\K_n,\,\J_n,\,n=0,\ldots,N+1$ determines all the coefficients of the matrix $M$.  It  follows from the Structure Theorem \ref{thm.structure} that:
\begin{equation}
\label{eqn.coupled}
\left[ \begin{array}{cccccccc}
\beta & \gamma & 0 & T & 0 & 0 & 0 & 0
\end{array}\right]\Tt = Mv\Tt.
\end{equation}

Any combination of $\beta,\gamma,T\ge 0$ and $\nu\ge 0$ that solves (\ref{eqn.coupled}) presents an optimal solution $\nu$ with the base sequence $\K_n,\,\J_n,\,n=0,\ldots,N+1$, and with $\beta,\gamma,T$ in the validity region of the base sequence.  

One can immediately see that for any $\theta>0$,  $\K_n,\,\J_n,\,n=0,\ldots,N+1$ is the optimal base sequence for the boundary values $(\theta\beta,\theta \gamma, \theta T)$, with the optimal solution $\theta \nu$.
Similarly, if $(\beta_*, \gamma_*, T_*)$ is in the validity region, with solution $\nu_*$, then 
$(\beta, \gamma, T)+(\beta_*, \gamma_*, T_*)$ is in the validity region, with the solution $\nu+\nu_*$.
Hence $\mathcal{T} $ is a convex cone.

Furthermore, the image under the linear transformation presented by $M$ of the convex non-negative orthant,  intersected with the planes $0$ at all the coordinates except $\beta, \gamma,T$ and intersected with $T\ge 0$, is the validity region of the base sequence.  This is obviously a convex polyhedral cone.
\end{proof}

The following corollary  follows immediate from Theorem \ref{thm.conves}.
\begin{corollary}
Let  $\ell(\theta)=(\beta(\theta), \gamma(\theta), T(\theta)) = (\beta, \gamma, T) + \theta(\delta\beta, \delta\gamma, \delta T)$ be a straight line of boundary parameters. As $\theta$ changes, within the validity region of a single base-sequence, each of the interval lengths $\tau_i$, each of boundary values $\bu^0_j, \bu^N_j, \bx^0_k, \bx^N_k, \bp^N_k, \bp^0_k,$ $\bq^N_j, \bq^0_j$ and each of $x_k(t_n), q_j(T(\theta)-t_n)$ are affine functions of $\theta$. 
\end{corollary}
\begin{proof}
The proof for $\tau_i, x_k(t_n), q_j(T(\theta)-t_n)$ is the same as in the proof of  Theorem 5 in \cite{weiss:08}.  The proof for the boundary values follows from the proof of  convexity  in  Theorem \ref{thm.conves}.  
\end{proof}

\section{Strong duality without Slater-type condition}\leavevmode
\label{sec.strong}
\begin{theorem}
\label{thm.strong}
If M-CLP and M-CLP$^*$ are feasible, then both have optimal solutions, and there is no duality gap.
\end{theorem}
\begin{proof}
We denote by M-CLP$(\beta, \gamma, b, c, T)$ the M-CLP problem with parameters $\beta, \gamma, b,$ $ c, T$.
Assume M-CLP/M-CLP$^*(\beta, \gamma, b, c, T)$ are feasible. 
Assume the Slater type condition \ref{asm.slater} does not hold, otherwise there is nothing to prove. 
We will now construct an optimal solution for this problem.  We choose some $\alpha > 0$. 
For $0<\theta\le 1$, let $\beta^* (\theta) = \beta + \alpha \theta 1_K$, $\gamma^*(\theta) = \gamma - \alpha \theta 1_J$, where $1_K$, $1_J$ are vectors of 1's of appropriate dimension. If Non-Degeneracy Assumption \ref{asm.nondegI} does not hold we choose vectors $\epsilon, \delta$ such that for all $0 < \theta \le 1$ $b^*(\theta) = b + \epsilon \theta$ and $c^*(\theta) = c + \delta \theta$ are in general position to the matrices $[A\; I]$ and $[A\Tt \;I]$ respectively, and also $\alpha 1_K + \epsilon T > 0$, $\alpha 1_J - \delta T > 0$, otherwise $\epsilon, \delta = 0$. It is clear that M-CLP/M-CLP$^*(\beta^*(1), \gamma^*(1), b^*(1), c^*(1), T)$ are feasible and satisfy the Slater type condition,  and therefore have strongly dual optimal solutions by Theorem \ref{thm.strongduality}.  Furthermore, Non-Degeneracy Assumption \ref{asm.nondegI} holds for  $b^*(1), c^*(1)$, and hence by part (ii) of the proof of the Structure Theorem \ref{thm.structure} the solution can be represented by an optimal base sequence $\J_n, \K_n, n = 0, \dots, N +1$.  Consider now a parametric family of problems M-CLP/M-CLP$^*(\beta^*(\theta), \gamma^*(\theta), b^*(\theta), c^*(\theta), T)$  $1 \ge \theta \ge 0$. For every $\theta > 0$ M-CLP/M-CLP$^*(\beta^*(\theta), \gamma^*(\theta), b^*(\theta), c^*(\theta), T)$ will still be strictly feasible and will satisfy Non-De\-generacy Assumption \ref{asm.nondegI}, and hence have optimal solutions represented by an optimal base sequence.  Such a base sequence will be valid for some interval $[\theta_m,\theta_{m+1}]$ by Theorem \ref{thm.conves}. By part (i) of Theorem \ref{thm.solform} all bases in an optimal base sequence are distinct, and so the total number of optimal base sequences is finite. Therefore we have a finite partition $\theta_0 \ge 1 > \theta_1 > \dots > \theta_{M-1} > 0 \ge \theta_M$,  where each interval belongs to a validity region of some optimal base sequence.  The base sequence of the last interval is an optimal base sequence for   M-CLP$(\beta, \gamma, b, c, T)$, with an optimal solution.
\end{proof}

\paragraph{Completion of proof of Theorem \ref{thm.structure}(ii)} This now follows immediately from Theorem \ref{thm.strong} and its proof. \quad  $ \square$

\begin{corollary}
One of the following statements about M-CLP and M-CLP$^*$ is true: \\
(i) both are feasible and have optimal solution without duality gap, or \\
(ii) both are infeasible, or \\
(iii) one of the problems is infeasible and the other is unbounded. 
\end{corollary}
\begin{proof}
Consider the Test-LP problem (\ref{eqn.ptestLP}) and its dual Test-LP$^*$. For Test-LP/Test-LP$^*$ we have the following three cases: \\
(i) Test-LP and Test-LP$^*$ are feasible. Then by Theorem \ref{the.feas} M-CLP and M-CLP$^*$ are also feasible and hence by Theorem \ref{thm.strong} have optimal solutions without duality gap. \\
(ii) Test-LP and Test-LP$^*$ are both infeasible. Then by Theorem \ref{the.feas} M-CLP and M-CLP$^*$ are also both infeasible. \\
(iii) One of Test-LP/Test-LP$^*$ problems is infeasible, while the other is unbounded, say Test-LP$^*$ is infeasible while Test-LP is unbounded.   Following the arguments of \cite{shindin-weiss:13} the same will be true for any discretization (see problem (10) in \cite{shindin-weiss:13}), and it follows that M-CLP is unbounded.
\end{proof}

\section{Relations between M-CLP and SCLP}
\label{sec.sclp}

We now discuss the relation between SCLP (separated continuous linear programs) and M-CLP. 
In \cite{weiss:08}  SCLP is formulated as: 
\begin{eqnarray}
\label{eqn.PWSCLP}
&\max & \int_0^T (\gamma\Tt + (T-t)c\Tt) u(t) + d\Tt x(t) \,dt   \nonumber   \hspace{1.0in} \\
\mbox{SCLP}  & \mbox{s.t.} &  \int_0^t G\, u(s)\,ds  + F x(t) \le \alpha + a t \\
 && \quad\; H u(t) \le b \nonumber \\
&& \quad x(t), u(t)\ge 0, \quad 0\le t \le T.  \nonumber
\end{eqnarray}
with a symmetric dual  
\begin{eqnarray}
\label{eqn.DWSCLP}
&\min & \int_0^T (\alpha\Tt + (T-t)a\Tt) p(t) + b\Tt q(t) \,dt     \nonumber  \hspace{0.9in} \\
\mbox{SCLP$^*$}  &\mbox{s.t.} &  \int_0^t  G\Tt\, p(s)\,ds + H\Tt q(t) \ge \gamma + c t \\
 && \quad\; F\Tt p(t) \ge d \nonumber \\
&& \quad q(t), p(t)\ge 0, \quad 0\le t \le T.   \nonumber
\end{eqnarray}
and  constant vectors and matrices $G,F,H,\alpha,a,b,\gamma,c,d$.

These problems do not in general satisfy strong duality.  Examples are given in Section \ref{sec.SCLP}.

%
%The simplex type algorithm of Weiss will solve any pair of problems (\ref{eqn.PWSCLP}), (\ref{eqn.DWSCLP}) which possess optimal solutions $u(t),p(t)$ that are bounded measurable functions, with $x(t),q(t)$ absolutely continuous.  However, there exist problems for which both (\ref{eqn.PWSCLP}) and (\ref{eqn.DWSCLP}) are feasible but either (\ref{eqn.PWSCLP}) or (\ref{eqn.DWSCLP}) or both do not possess optimal solutions $u(t),p(t)$ in the space of bounded measurable functions.  Moreover,
%(\ref{eqn.PWSCLP}) could be feasible and even posses optimal solution in the space of bounded measurable function,
%while (\ref{eqn.DWSCLP}) is infeasible. Example of corresponding problems are shown in section \ref{sec.examples}. Such problems
%cannot be solved by the algorithm of Weiss. This raises the question whether they can be solved in
%the space of measures and motivate our research.

\begin{definition}[Definition 2.1 in \cite{shindin-weiss:13}]
Consider the SCLP problem (\ref{eqn.PWSCLP}). Then the M-CLP problem with the following data:
\begin{equation*}
\begin{array}{c}
A = \left[ \begin{array}{cccc} G & 0 & F & -F \\   0 & 0 & -I & I \\  
H & I & 0 & 0  \\  -H & -I & 0 & 0 \end{array}  \right] \quad
U(t) = \left[  \begin{array}{c} U_*(t) \\  U_s(t) \\ U^{+}(t) \\ U^{-}(t)  \end{array}  \right] \quad
\beta^* + b^*t =  \left[  \begin{array}{c} \alpha \\  0 \\ 0 \\ 0  \end{array}  \right] +
 \left[ \begin{array}{c} a \\  0 \\ b \\ -b  \end{array} \right] t, \\
\gamma^* + (T-t)c^* = \left[ \gamma \quad 0 \quad\;  d  \quad -d \right] + (T-t)\left[ c \quad 0 \quad\;  0  \quad 0 \right]
\end{array} 
\end{equation*}
is called the M-CLP extension of SCLP. \/ M-CLP$^*$ extension of SCLP$^*$ is defined similarly.
\end{definition}

The following theorem extends Theorem 2.2 in \cite{shindin-weiss:13}.
\begin{theorem}
\label{thm.generalization}
M-CLP/M-CLP$^*$ are generalizations of  SCLP/SCLP$^*$ (\ref{eqn.PWSCLP}),(\ref{eqn.DWSCLP}) in the  following sense: \\
(i) if SCLP is feasible then M-CLP extension is feasible, equivalently, if M-CLP extension is infeasible then SCLP is infeasible \\
(ii) if SCLP is feasible and M-CLP$^*$ extension is feasible, then the supremum of the objective of SCLP is equal to the objective value of the optimal solution of the M-CLP extension \\
(iii) if SCLP is feasible and M-CLP$^*$ extension is infeasible, then SCLP is unbounded \\
(iv)  SCLP has an optimal solution if and only if  M-CLP has an absolutely continuous optimal solution, with possible jumps in $U^+(0), U^-(0)$. 
\end{theorem}
\begin{proof}
(i)  Let $u(t), x(t)$ be a feasible solution of SCLP. We define the following:
\[\tU_*(t) = \frac{t}{T} \int_0^T u(t)dt, \quad \tU_s(t) = \left(b - H \frac{1}{T} \int_0^T u(t)dt\right)t, \quad \tU_f(t)= x(0) + \frac{t}{T}(x(T) - x(0)) \]
One can see that $\tU_*(t), \tU_s(t)$ is non-decreasing functions, while $\tU_f(t)$ is monotonic in each coordinate,  and hence could be represented as $\tU_f(t) = \tU^+(t) - \tU^-(t)$, where $\tU^+(t),  \tU^-(t)$ are non-decreasing functions. Then the resulting $\tU(t) = [\tU_*(t)\; \tU_s(t)\; \tU^+(t)\; \tU^-(t)]$ satisfies the constraints of the M-CLP extension.

(ii) The proof is the same as the proof of Theorem 2.2(iii) in \cite{shindin-weiss:13}, 
except that  existence of optimal solution for M-CLP/M-CLP$^*$ and its modifications is guaranteed without the Slater-type condition.

(iii) Assume to the contrary, that the objective of SCLP has an upper bound $\Psi$. 
Let  $u^*(t), x^*(t)$ be feasible solution of SCLP. We consider now the following discretization of the M-CLP extension of SCLP, which we call dSCLP:
\begin{equation}
\label{eqn.dSCLP}
\!\!\!\begin{array}{ll}
\max &  \left(\gamma^* + c^*(T - \theta)\right)\Tt 2\theta u^1 + \left(\gamma^* + c^* \left(\frac{T}{2}-2\theta\right)\right)\Tt (T -4\theta) u^2 +(\gamma^* + c^* \theta)\Tt 2\theta u^3 \\
\mbox{s.t.} & 2 \theta A  u^1 \le \beta^* + 2\theta b^* - [F\Tt\; 0]\Tt x^*(0) \\
 & 2 \theta A  u^1 + (T - 4\theta) A u^2 \le \beta^* + b^* (T-2\theta) - [F\Tt\; 0]\Tt x^*(0) \\
 & 2 \theta A  u^1 + (T - 4\theta) A u^2 + 2 \theta A  u^3 \le \beta^* + b^* T - [F\Tt\; 0]\Tt x^*(0) \\
 & u^1, u^2, u^3 \ge 0
\end{array} 
\end{equation}
where $\theta$ is a small number defined below. 

Taking a feasible solution of the SCLP, extending it to a feasible solution of the M-CLP extension as described in proof of (i) and setting:
\begin{eqnarray*}
&& \tu^1 = \frac{\tU(2\theta) - \tU(0)}{2\theta} \quad \tu^2 = \frac{\tU(T-2\theta) -\tU(2\theta)}{T - 4\theta} \quad \tu^3 = \frac{\tU(T) - \tU(T-2\theta)}{2\theta}
\end{eqnarray*} 
one can obtain a feasible solution of dSCLP (\ref{eqn.dSCLP}). Conversely, any feasible solution of dSCLP can be extended to a feasible solution of the M-CLP extension: 
\begin{equation}
 \tU(t) = \left\lbrace \begin{array}{ll} \displaystyle [I\; 0]\Tt x^*(0) + u^1 t  & t \le 2\theta \\
\displaystyle [I\; 0]\Tt x^*(0) + 2\theta (u^1 - u^2) + u^2 t  & 2\theta < t \le T - 2\theta  \\
\displaystyle [I\; 0]\Tt x^*(0) + 2\theta (u^1 - u^3) + (T-4\theta)(u^2 -u^3)  + u^3 t &  t > T - 2\theta  
 \end{array}  \right.  
\end{equation}
%\begin{equation}
%\begin{array}{l}
%\begin{array}{l} \displaystyle u = [u_*\, u_s\, u^+\, u^-] \\
%\displaystyle U = [U_*\, U_s\, U^+\, U^-]
%\end{array} \quad
%\tu(t) = \left\lbrace \begin{array}{ll} \displaystyle u_*  & t \le 2\theta \\
%\displaystyle U_* & t > 2\theta  
% \end{array}  \right. \\
% \tx(t) = \left\lbrace \begin{array}{ll} \displaystyle  x^*(0) + (u_+ - u_-) t  & t \le 2\theta \\
%\displaystyle x^*(0) + 2\theta (u_+ - u_- - U_+ + U_-) + (U_+ - U_-) t & t > 2\theta  
% \end{array}  \right. \end{array}  
%\end{equation}
This solution can be further extended to a feasible solution of SCLP, by setting $\tu(t) = \frac{d\tU_*(t)}{dt}$ and $\tx(t) = \tU^+(t) -\tU^-(t)$.
Moreover, the corresponding  objective values of these solutions satisfy: $V(dSCLP) + T d\Tt x^*(0) = V(M-CLP) = V(SCLP)$.

We now write the constraints of the dual of dSCLP:
\begin{equation}
\label{eqn.constrd2}
\begin{array}{l}
A\Tt p^1 \ge \gamma^* + c^*\theta \\
A\Tt p^1 + A\Tt p^2 \ge \gamma^* + c^*(\frac{T}{2} - 2 \theta) \\
A\Tt p^1 + A\Tt p^2 + A\Tt p^3 \ge \gamma^* + c^*(T -\theta) \\
p^1, p^2, p^3 \ge 0
\end{array} 
\end{equation}
One can see that at $\theta=0$ constraints (\ref{eqn.constrd2}) are identical to the constraints of the discretization of M-CLP$^*$ and hence infeasible. From LP theory  it follows that one can find $\theta_1 > 0$ small enough, such that (\ref{eqn.constrd2}) still be infeasible for $\theta =\theta_1$. This determines our choice of $\theta$. 

Finally, because dSCLP is feasible and (\ref{eqn.constrd2}) is not, it follows that dSCLP is unbounded.  Thus  $V(dSCLP) > \Psi - T d\Tt x^*(0)$ and $V(SCLP) > \Psi$. This contradicts our assumption and hence SCLP is unbounded.
%\begin{equation}
%\label{eqn.dSCLP*}
%\begin{array}{lll}
%& \min & (\beta^* + b^* T - [F\Tt\; 0]\Tt x^*(0))\Tt p + (\beta^* + 2 \theta b^*(T - \theta) - [F\Tt\; 0]\Tt x^*(0))\Tt P  \\
%& \mbox{s.t.} & A\Tt p \ge \gamma^* + c^*\theta \\
%\mbox{dSCLP$^*$} & & A\Tt p + A\Tt P \ge \gamma^* + c^*(T - \theta) \\
%& & p, P \ge 0
%\end{array} 
%\end{equation}

(iv) {\em Sufficiency} is already shown in Theorem 2.2(iii) in \cite{shindin-weiss:13}. \\
{\em Necessity:} Let $u^*(t), x^*(t)$ be an optimal solution of SCLP. Then, we can find an absolutely continuous function $x^{**}(t)$, such that $u^*(t), x^{**}(t)$ is also optimal solution of SCLP (see Theorem 7.8 in \cite{anderson-nash:87}). We now write $x^{**}(t)$ as the difference of two non-decreasing absolutely continuous functions $x^{**}(t) = \tU^+(t) - \tU^-(t)$. Let $u_s(t)$ be the slacks of the constraints $H u(t)\le b$, and let $\tU_*(t) = \int_0^t u^*(t) dt, U_s(t) = \int_0^t u_s(t)dt$. Then the resulting $\tU(t) = [\tU_*(t), \tU_s(t), \tU^+(t), \tU^-(t)]$ satisfies the constraints of the M-CLP extension, with the same objective value and hence it is optimal for the M-CLP extension by (ii) and (iii). This solution is indeed absolutely continuous, except possible for a jump at $\tU^+(0), \tU^-(0)$
\end{proof}

%For subsets $\K,\J$ we define Rates-LP/LP$^*$  with sign restrictions corresponding to $\K,\J$ as Rates-LP$(\K,\J)$/Rates-LP$^*(\K,\J)$
%
%
%Let $\tilde{u}^0,\tilde{x}^0,\tilde{p}^N,\tilde{q}^N$ be any feasible solution of first boundary equations (\ref{eqn.boundary1}). We denote SCLP (\ref{eqn.PWSCLP}), SCLP$^*$  (\ref{eqn.DWSCLP}) problems with data $G=A, H=\emptyset, F=\emptyset, \tilde{\alpha}=\tilde{x}^0, \tilde{a}=b,\tilde{b}=\emptyset, \tilde{\gamma}=- \tilde{q}^N, \tilde{c}=c, d=\emptyset, \tilde{T}=T$ as {\em internal SCLP/SCLP$^*(\tilde{x}^0,\tilde{q}^N,\tilde{T})$}.

%===============================================================
%Illustrative examples
%===============================================================
\section{Illustrative examples}
\label{sec.examples}
Here we give some examples that motivate our research and describe important properties of M-CLP solutions.  In Section \ref{sec.onedim} we describe all possible solutions for  one dimensional M-CLP problems, with $K=J=1$.   
In Section \ref{sec.twothreedim} we present an example of an M-CLP problems with $K,J=2$, which we  solve for all time horizons $0<T<\infty$, to illustrate how the solution evolves over  changing time horizons.  Finally, in Section \ref{sec.SCLP} we present examples of SCLP problems, which do not satisfy strong duality with SCLP$^*$. 
%Examples \ref{ex2.1} and \ref{ex3.1} are examples in which no single interval solution exists, 
%and we present the unique solutions valid for small time horizons, $0<T<\theta $, consisting of two intervals in Example \ref{ex2.1}, and of three intervals in Example \ref{ex3.1}.  We note that all the examples which we have solved so far have solutions that are as described in the Structure Theorem \ref{thm.structure}.

%%%%%%%%%%%%%%%%%%%%%%%%%%%%%%%%%%%%%%%%%%
%One-dimensional cases
%%%%%%%%%%%%%%%%%%%%%%%%%%%%%%%%%%%%%%%%%%%
\subsection{One-dimensional cases}
\label{sec.onedim}
Let $J\!=\!K\!=\!1$, so that  M-CLP/M-CLP$^*$ are parametrized by the five values of $a,\,\beta,\,b,\,\gamma,\,c$.  The nature of the solution depends only on the signs of these parameters, so there are $32$ cases.  Because the dual for $a,\,\beta,\,b,\,\gamma,\,c$ is in fact the primal problem with $-a,-\gamma,-c,-\beta,-b$, we get all the different cases by considering the primal and the dual solutions for $a>0$ and the 16 sign combinations of $\beta,\,b,\,\gamma,\,c$.  Without loss of generality we rescale the problem and set $a=1$.  We will consider all 16 cases, and describe the solutions as a function of $T$.  

Even with these very small one dimensional examples we will demonstrate many of the features of the solutions, as described in the structure theorem:
\begin{itemize}
\item
Four of the cases have feasible primal and dual solutions for all $0\le T < \infty$.
\item
Four of the cases have feasible primal and dual solutions for $0\le  T \le \overline{T}$ for some finite $\overline{T}$ but are primal infeasible with unbounded duals for $T > \overline{T}$.
\item
Eight of the cases are primal infeasible for all $T$, with unbounded duals.
%\item
%Whenever primal and dual are feasible the solution is of the form given by the structure theorem.
\item
Impulse controls are sometimes used, both at times $0$ and $T$.
\item
The boundary values may be constant or they may vary linearly with $T$.
\item
The state variables may be discontinuous at time $T$, in which case they have a jump down to $0$.
\item
When the Uniqueness Condition \ref{asm.uniqueness} does not hold the solution may be non-unique.
\item
Solutions for small $T$ are single interval solutions for all cases.  
For larger $T$ they may require two intervals.  
\item 
For one dimensional problem it is not possible that both primal and dual are infeasible.
%An example for which there is no single interval solution even for small $T$ will be given in Section \ref{sec.twothreedim}, as it requires dimension 2.
\end{itemize}

We now describe the solution for the various cases.   
We give a brief description of the solutions in the eight feasible cases, and display the solutions in 
Figures  \ref{fig.case1}-- \ref{fig.case8}. 
In each figure we plot the right hand side of the primal, the objective coefficients of the primal, the primal solution, and the dual solution which is running in reversed time.  In the plots of the solutions we draw the cumulative controls $U(t),\,P(t)$ where we put an up-arrow to denote impulses, and the states (the slacks) $x(t),\,q(t)$.

\subsubsection*{Feasible for all $T$}  When $\beta>0,\,b>0$ the primal and dual are feasible and bounded for all $0\le T<\infty$.
\subsubsection*{Case 1:  $\beta>0,\,b>0,\,\gamma>0,\,c>0$}

%{\raggedright
%$\displaystyle \bu^0 = \beta, x^0=0, \, \bp^0=0, \bq^0=0;\; u =b;\, \dx=0, p=c, \dq=0;\; 
%\bu^N=0,\,\bx^N=0,\,\bp^N = \gamma,\,q^N = 0.$  \\}
Here the optimal primal solution is to use $U(t)=\beta+bt$, with an impulse at $0$, and the optimal dual solution is to use $P(t)=\gamma+ct$ with an impulse at $0$.  The slacks are $x(t)=q(t)=0$.

\begin{figure}[h!] 
 \centering

 \includegraphics[width=5in]{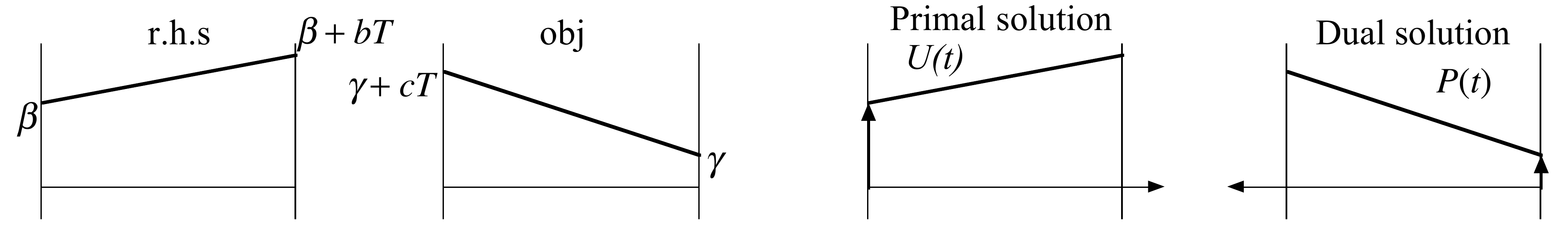} 
  \caption{$K=J=1$, Case 1:  $\beta  > 0,b > 0,\gamma  > 0,c > 0$}

  \label{fig.case1}
\end{figure}

\subsubsection*{Case 2:  $\beta>0,\,b>0,\,\gamma>0,\,c<0$}
%{\raggedright
%$ \displaystyle \bu^0 = 0, x^0=\beta, \, \bp^0=0, \bq^0=-cT; \; u =0,\, \dx=b, p=0, \dq=-c;$ \\
%$ \displaystyle x^\downarrow= \beta+bT;\; \bu^N=\beta+bT,\,\bx^N=0,\,\bp^N = \gamma,\,q^N = 0.$
% \\}
Here the optimal primal solution is to use a single impulse at time $T$, so $U(t)=0,\,t<T$ and $U(T)=\beta+bT$.  The slack $x(t)$ is discontinuous at $T$:  $x(t)=\beta+bt,\,0\le t<T$, $x(T)=0$, with a jump 
$x(T^-)-x(T)=\beta+bT$.   The optimal dual solution is to have $P(t)=\gamma$ with a single impulse at $0$, and $q(t)=- c t$.
\begin{figure}[h!] 
 \centering
 \includegraphics[width=5in]{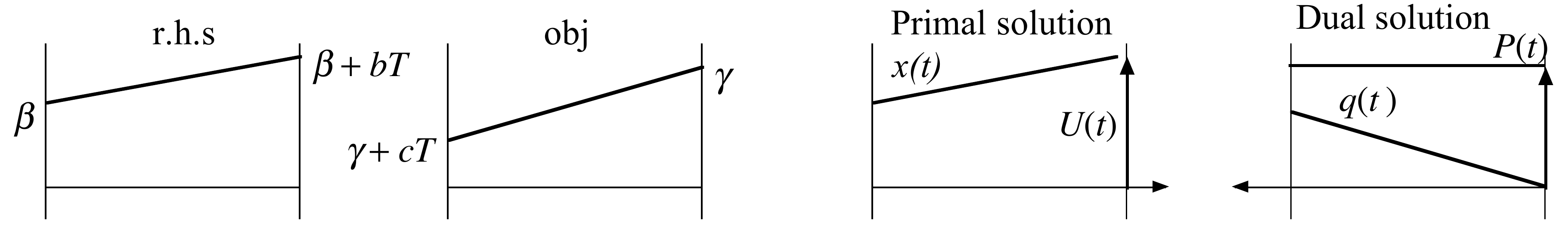} 
  \caption{$K=J=1$, Case 2:  $\beta  > 0,b > 0,\gamma  > 0,c < 0$}
  \label{fig.case2}
\end{figure}

\subsubsection*{Case 3:  $\beta>0,\,b>0,\,\gamma<0,\,c>0$}
%For time horizon $0\le T <  - \frac{\gamma }{c} $:
%\newline
%{\raggedright
%$ \displaystyle \bu^0 = 0, x^0=\beta, \, \bp^0=0, \bq^0=-\gamma -cT; \; u =0,\, \dx=b, p=0, \dq=-c;$\\
%$  \bu^N=0,\,\bx^N=\beta+bT,\,\bp^N = 0,\,q^N = -\gamma.$}
%\newline
%For time horizon $T=  - \frac{\gamma }{c} $  (for any $0\le\theta\le1$):
%\newline
%{\raggedright
%$ \displaystyle \bu^0 = \theta \beta,\, x^0=(1-\theta) \beta, \, \bp^0=0, \bq^0=0;$\\
%$u =0,\, \dx=b, p=0, \dq=-c;$ \\
%$  \bu^N=0,\,\bx^N=(1-\theta) \beta+bT,\,\bp^N = 0,\,q^N = -\gamma.$}
%\newline
%For time horizon $  - \frac{\gamma }{c} < T < \infty$:
%\newline
%{\raggedright
%$ \displaystyle \bu^0 = \beta, x^0=0, \, \bp^0=0, \bq^0=0;$\\
%$\; u^1 =b,\, \dx^1=0, p^1=c, \dq^1=0;\; u^2=0,\dx^2=b,p^2=0,\dq^2=-c;$\\
%$  \bu^N=0,\,\bx^N=b(T+ \frac{\gamma }{c}) ,\,\bp^N = 0,\,q^N = -\gamma.$}

The optimal solution for this case consists of a single interval when  $T<- \frac{\gamma}{c}$ and of two intervals when $T>- \frac{\gamma}{c}$.  It is non-unique when $T = - \frac{\gamma}{c}$.  For time horizon $T<- \frac{\gamma}{c}$ the solution is to have $U(t)=P(t)=0$  with the slacks $x(t)=\beta+bt$, $q(t)=-\gamma - ct$, for $0\le t \le T$.
The solution for $T>- \frac{\gamma}{c}$  has $U(t)=\beta+bt,\,0\le t \le T + \frac{\gamma}{c}$, with an impulse at $0$, and $U(t)=U(T + \frac{\gamma}{c}),\,t>T + \frac{\gamma}{c}$, with the slack  $x(t)= 0,\,t< T + \frac{\gamma}{c}$, and $x(t)= b(t- T - \frac{\gamma}{c}),\,t>T + \frac{\gamma}{c}$.  The dual solution (in reversed time) has $P(t)=0, t\le - \frac{\gamma}{c}$ and $P(t)=c(t + \frac{\gamma}{c}),\,t>- \frac{\gamma}{c}$, with the slack $q(t)=-\gamma - ct,\,t< - \frac{\gamma}{c}$, and $q(t)=0,\,t\ge - \frac{\gamma}{c}$.

At the time horizon $T=- \frac{\gamma}{c}$ the solution to the primal problem is not unique:  We can have $U$ consisting of a single impulse at time $0$, of size $\bu^0$ that can take any value from $0$ to $\beta$, so $U(t)=\bu^0$, and the corresponding slack is $x(t)=\beta+bt - \bu^0$, for $0\le t\le T$.
Note that at $T=- \frac{\gamma}{c}$ the vector $\left[ \begin{array}{c}\gamma \\ \gamma + c T \end{array} \right]  = \left[ \begin{array}{c}\gamma \\ 0 \end{array} \right]$
 is not in general position to the matrix   $\left[ \begin{array}{cccc} A\Tt & 0 & I & 0 \\ A\Tt & A\Tt & 0 & I \end{array} \right]$, so the Uniqueness Condition \ref{asm.uniqueness} is indeed violated.
\begin{figure}[H] 
 \centering
 \includegraphics[width=5in]{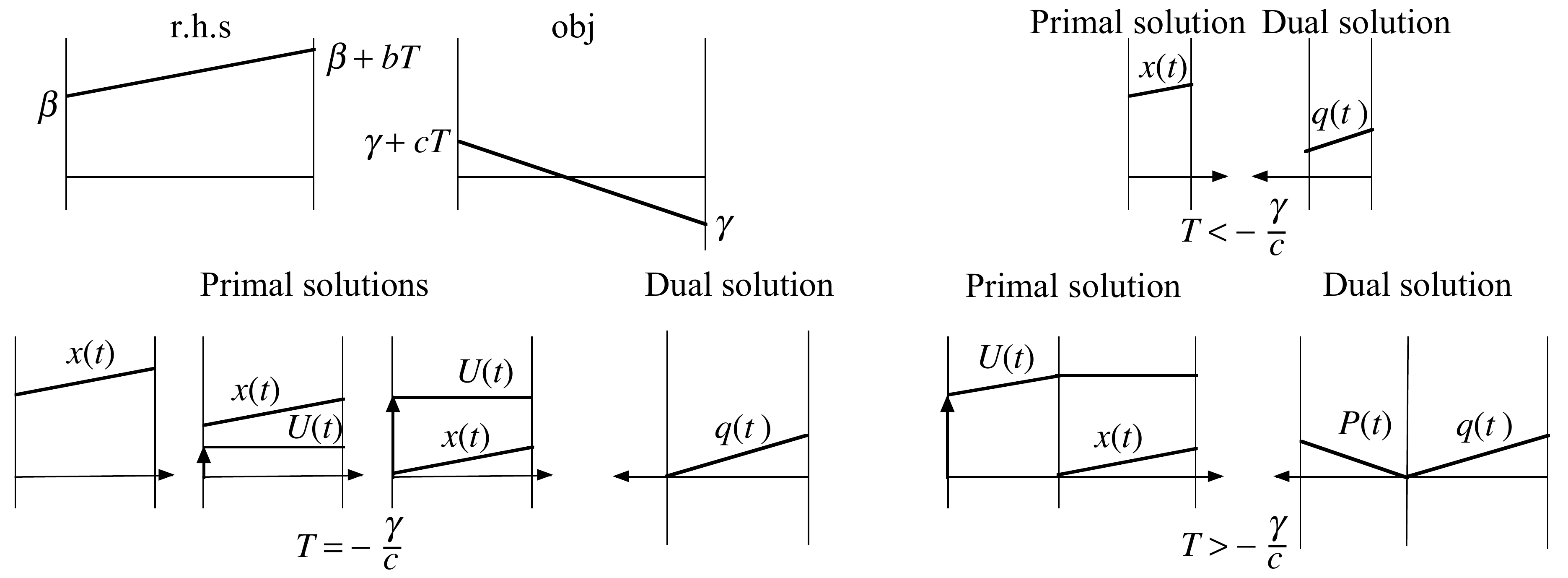} 
  \caption{$K=J=1$, Case 3:  $\beta  > 0,b > 0,\gamma  < 0,c > 0$}

  \label{fig.case3}
\end{figure}

\subsubsection*{Case 4:  $\beta>0,\,b>0,\,\gamma<0,\,c<0$}
%{\raggedright
%$ \displaystyle \bu^0 = 0, x^0=\beta, \, \bp^0=0, \bq^0=-\gamma - cT; \; u =0,\, \dx=b, p=0, \dq=-c;$\\
%$ \bu^N=0,\,\bx^N=\beta+bT,\,\bp^N =0,\,q^N = -\gamma.$ \\}
Here the solution is the same as for Case 3 when $T<- \frac{\gamma}{c}$:  the primal and dual controls are $U(t)=P(t)=0$, with slacks $x(t)=\beta+bt$, $q(t)=-\gamma-ct$.
\begin{figure}[ht!] 
 \centering
 \includegraphics[width=5in]{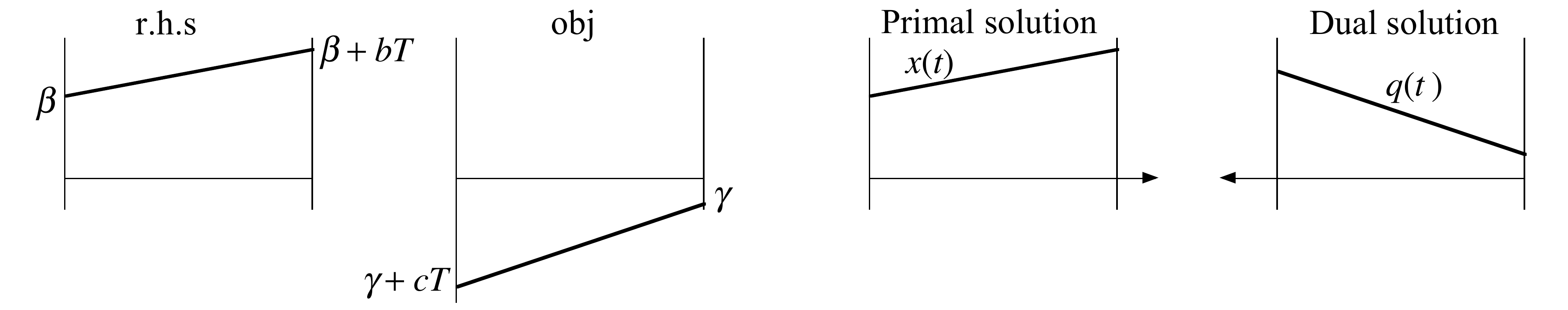} 
  \caption{$K=J=1$, Case 4:    $\beta  > 0,b > 0,\gamma  < 0,c < 0$}
  \label{fig.case4}
\end{figure}

\subsubsection*{Feasible for some $T$}  
The next 4 cases have $\beta>0,\,b<0$ and are feasible and bounded only for $0\le T \le -\frac{\beta}{b}$.  For $T >  -\frac{\beta}{b}$ the primal is infeasible and the dual unbounded.

\subsubsection*{Case 5:  $\beta>0,\,b<0,\,\gamma>0,\,c>0$}
%{\raggedright
%$ \displaystyle \bu^0 = \beta+bT, x^0=-bT, \, \bp^0=0, \bq^0=0; \; u = 0,\, \dx= b, p=0, \dq=-c;$\\
%$ \bu^N=0,\,\bx^N=0,\,\bp^N =\gamma+cT,\,q^N = cT.$ \\}
Here the primal  control is constant $U(t)=\beta+bT$ with impulse at $0$, and slack $x(t)=- b(T-t)$, and the dual control is constant $P(t)=\gamma+cT$ with impulse at $0$, and slack $q(t)=c(T-t)$.
\begin{figure}[h!] 
 \centering
 \includegraphics[width=5in]{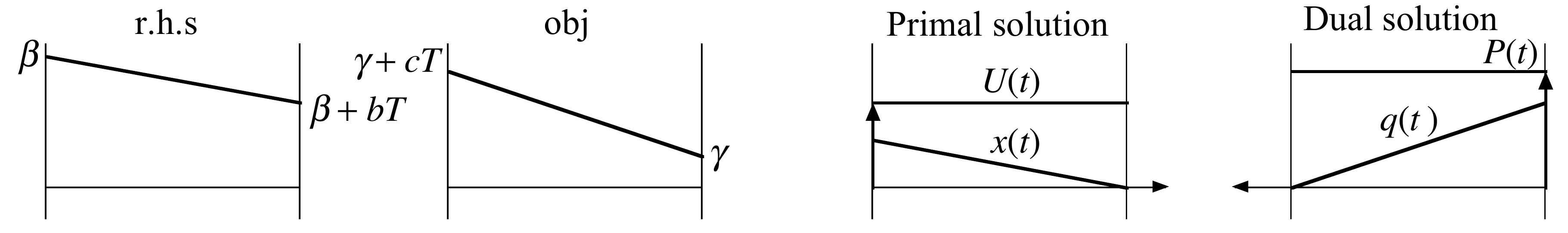} 
  \caption{$K=J=1$, Case 5:    $\beta  > 0,b < 0,\gamma  > 0,c > 0$}
  \label{fig.case5}
\end{figure}

\subsubsection*{Case 6:  $\beta>0,\,b<0,\,\gamma>0,\,c<0$}
%{\raggedright
%$ \displaystyle \bu^0 =0, x^0=\beta, \, \bp^0=0, \bq^0=-cT; \; u = 0,\, \dx= b, p=0, \dq=-c;$\\
%$x^N=\beta+bT,\; \bu^N=\beta+bT,\,\bx^N=0,\,\bp^N =\gamma,\,q^N =0.$ \\}
Here the primal  control is $U(t)=0,\,t<T,\,U(T)=\beta+bT$ with impulse at $T$, and slack $x(t)=\beta+ bt,\,t<T,\,x(T)=0$, with a downward jump $x^\downarrow=\beta+bT$  at $T$, similar to Case 2.  The dual control is $P(t)=\gamma$, with an impulse at $0$, and  slack $q(t)=-ct$.
\begin{figure}[h!] 
 \centering
 \includegraphics[width=5in]{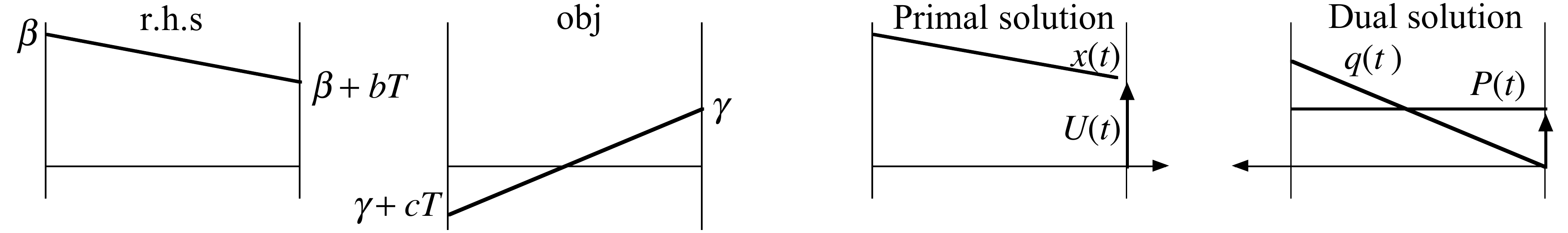} 
  \caption{$K=J=1$, Case 6:    $\beta  > 0,b < 0,\gamma  > 0,c < 0$}
  \label{fig.case6}
\end{figure}

\subsubsection*{Case 7:  $\beta>0,\,b<0,\,\gamma<0,\,c>0$}
%For time horizon $0\le T <  - \frac{\gamma }{c} $:
%\newline
%%{\raggedright
%%$ \displaystyle \bu^0 = 0, x^0=\beta, \, \bp^0=0, \bq^0=-\gamma -cT; \; u =0,\, \dx=b, p=0, \dq=-c;$\\
%%$  \bu^N=0,\,\bx^N=\beta+bT,\,\bp^N = 0,\,q^N = -\gamma.$}
%\newline
%For time horizon $T=  - \frac{\gamma }{c} $  (for any $0\le\theta\le1$):
%\newline
%{\raggedright
%$ \displaystyle \bu^0 = \theta( \beta+bT),\, x^0=\beta -\theta(\beta+bT), \, \bp^0=0, \bq^0=0;$\\
%$u =0,\, \dx=b, p=0, \dq=-c;$ \\
%$  \bu^N=0,\,\bx^N=(1-\theta)( \beta+bT),\,\bp^N = 0,\,q^N = -\gamma.$}
%\newline
%For time horizon $  - \frac{\gamma }{c} < T \le - \frac{\beta }{b}$:
%\newline
%{\raggedright
%$ \displaystyle \bu^0 = \beta+bT, x^0=-bT, \, \bp^0=0, \bq^0=0;\; u =0,\, \dx=b, p=0, \dq=-c;$\\
%$  \bu^N=0,\,\bx^N=0 ,\,\bp^N = \gamma+cT,\,q^N = cT.$}
The optimal solution for this case consists of a single interval for all $0\le T\le -\frac{\beta}{b}$, but it is different when $0 \le T\le -\frac{\gamma}{c}$, and when      $-\frac{\gamma}{c}\le T \le -\frac{\beta}{b}$.
For   $T<- \frac{\gamma}{c}$ the controls are $U(t)=P(t)=0$ with slacks $x(t)=\beta+bt$, $q(t)=-\gamma-ct$.  For $- \frac{\gamma}{c}<T\le -\frac{\beta}{b}$  The primal control is $U(t)=\beta+bT$ with impulse at $0$ and slack $x(t)=-b(T-t)$, and the dual control is $P(t)=\gamma+cT$ with impulse at $0$ and slack $q(t)=-c(T-t)$.   
At the time horizon $T=- \frac{\gamma}{c}$ the Uniqueness Condition \ref{asm.uniqueness} is violated, and the solution to the primal problem is not unique (similar to Case 3):  There is an impulse $\bu^0$ of size $0\le \bu^0\le \beta - b \frac{\gamma}{c}$, with $U(t)=\bu^0$, and slack  $x(t) = \beta-\bu^0  +bt$.  
\begin{figure}[h!] 
 \centering
 \includegraphics[width=5in]{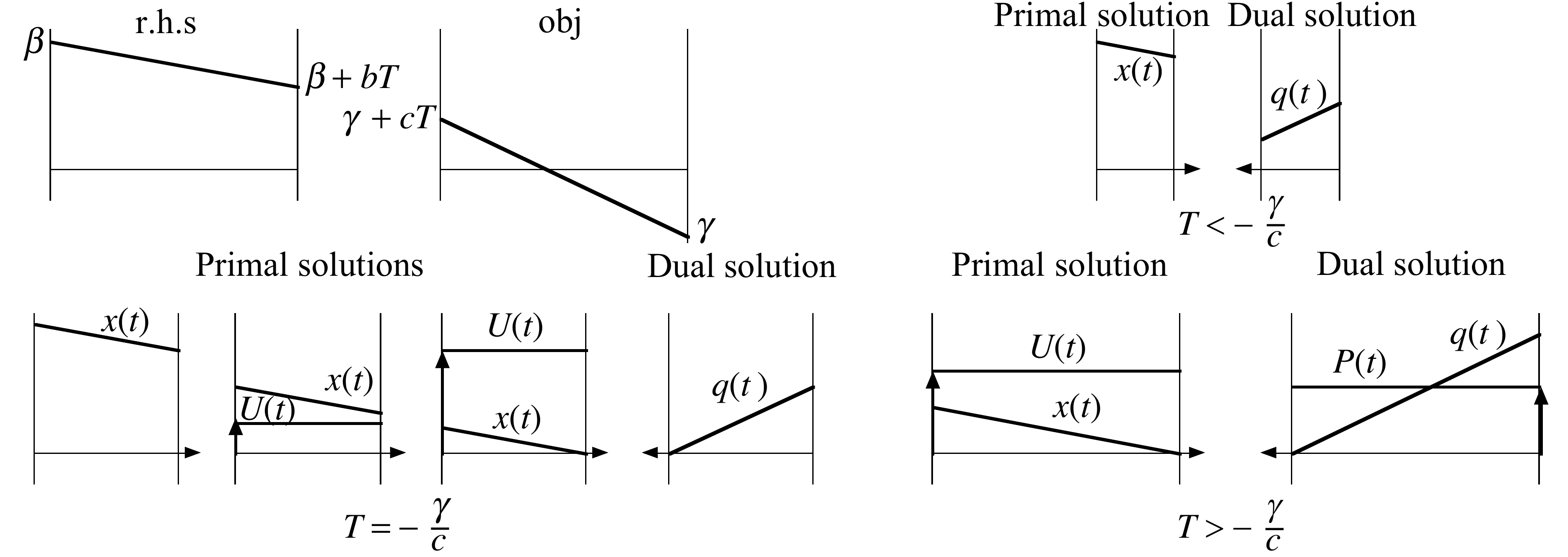} 
  \caption{$K=J=1$, Case 7:    $\beta  > 0,b < 0,\gamma  < 0,c > 0$}
  \label{fig.case7}
\end{figure}

\subsubsection*{Case 8:  $\beta>0,\,b<0,\,\gamma<0,\,c<0$}
%{\raggedright
%$ \displaystyle \bu^0 =0, x^0=\beta, \, \bp^0=0, \bq^0=-\gamma-cT; \; u = 0,\, \dx= b, p=0, \dq=-c;$\\
%$ \bu^N=0,\,\bx^N=\beta+bT,\,\bp^N =0,\,q^N =-\gamma.$ \\}
This is similar to Case 4.   The primal  and dual controls are $U(t)=P(t)=0$ and the primal and dual slacks are  $x(t)=\beta+bt$, $q(t)=-\gamma-cT$.
\begin{figure}[h!] 
 \centering
 \includegraphics[width=5in]{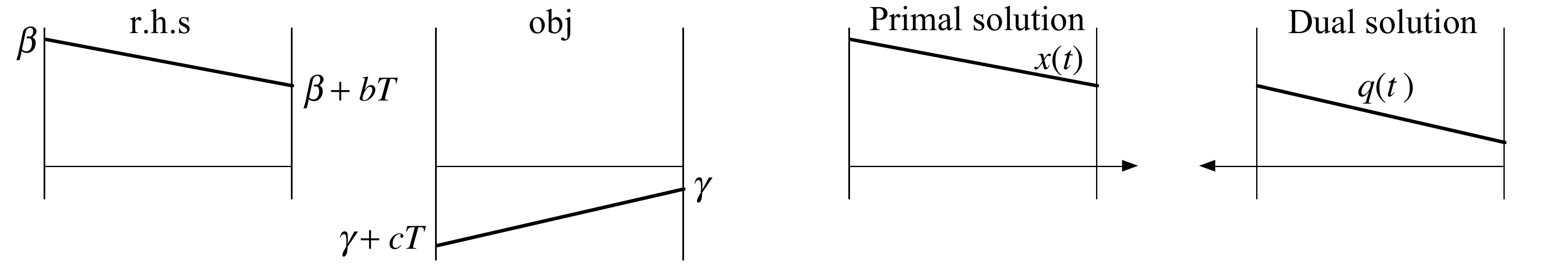} 
  \caption{$K=J=1$, Case 8:    $\beta  > 0,b < 0,\gamma  < 0,c < 0$}
  \label{fig.case8}
\end{figure}

\subsubsection*{Infeasible for all $T$}  
When $\beta<0$ the primal problem is infeasible, and the dual problem is unbounded, for all 8 sign combinations of $b,\,\gamma,\,c$.

%%%%%%%%%%%%%%%%%%%%%%%%%%%%%%%%%%%%%%%%%%%%%%
%2$\times$2 and 3$\times$3 examples
%%%%%%%%%%%%%%%%%%%%%%%%%%%%%%%%%%%%%%%%%%%%%%
\subsection{2$\times$2 example}
\label{sec.twothreedim}
One-dimensional examples are too simple to show all the possible diversity of M-CLP solutions. For better illustration of this diversity we need to consider $J=K=2$. We present the solution of a 2$\times$2 problem which is primal and dual feasible for all $T$.  It illustrates how the solution evolves as the time horizon changes.  
%Example \ref{ex2.1}  demonstrates that without the additional complementarity condition one may have problems with no single interval solutions.  The final Example \ref{ex3.1} is a 3$\times$3 example which for small $T$ has a 3 interval solution.  

{\sc Example \ref{sec.twothreedim}\quad} The problem data are:
\[
 A=\left(
\begin{array}{cc}
 5 & 2 \\
 3 & 4
\end{array}
\right) \quad \beta=\left(
\begin{array}{c}
 8 \\
 10
\end{array}
\right)  \quad b=\left(
\begin{array}{c}
 3 \\
 1
\end{array}
\right) \quad \gamma=\left(
\begin{array}{c}
 5 \\
 6
\end{array}
\right)  \quad c=\left(
\begin{array}{c}
 1 \\
 2
\end{array}
\right)\]
Initially, for small $T$, this problem has a single interval solution.  It is given by:
%\[
%\K_0=\{2\},\,\J_0=\emptyset,\;\K_1=\{2\},\;\J_1=\{2\},\; \K_2=\{ 2\},\, \J_2= \emptyset.
%\]
{\raggedright
$ \displaystyle
\bu^0=\left[ \begin{array}{c}\displaystyle \frac{6}{7} + \frac{4}{35} T \\[0.3cm] \displaystyle \frac{13}{7} - \frac{2}{7} T \end{array} \right] \;
x^0 = \left[ \begin{array}{c}\displaystyle 0 \\[0.3cm] \displaystyle \frac{4T}{5} \end{array} \right]\;  \bp^0 =\left[ \begin{array}{c}\displaystyle 0 \\[0.1cm] \displaystyle 0 \end{array} \right] \; \bq^0 = \left[ \begin{array}{c}\displaystyle 0 \\[0.3cm] \displaystyle 0 \end{array} \right]$\\[0.1cm]
$ \displaystyle 
u^1 = \left[ \begin{array}{c}\displaystyle \frac{3}{5} \\[0.3cm] \displaystyle 0 \end{array} \right]\; \dx^1 =\left[ \begin{array}{c}\displaystyle 0 \\[0.3cm] \displaystyle - \frac{4}{5} \end{array} \right] \;  p^1 = \left[ \begin{array}{c}\displaystyle \frac{1}{5} \\[0.3cm] \displaystyle 0 \end{array} \right]\; \dq^1 = \left[ \begin{array}{c}\displaystyle 0 \\[0.3cm] \displaystyle - \frac{8}{5} \end{array} \right] $\\[0.1cm]
$\displaystyle
\bu^N = \left[ \begin{array}{c}\displaystyle 0 \\[0.3cm] \displaystyle 0 \end{array} \right]\; \bx^N = \left[ \begin{array}{c}\displaystyle 0 \\[0.3cm] \displaystyle 0 \end{array} \right]\; \bp^N = \left[ \begin{array}{c}\displaystyle \frac{1}{7} - \frac{12}{35} T \\[0.3cm] \displaystyle \frac{10}{7} + \frac{4}{7} T \end{array} \right]\; q^N = \left[ \begin{array}{c}\displaystyle 0 \\[0.3cm] \displaystyle \frac{8T}{5} \end{array} \right]$.\\[0.1cm] }
\noindent and it  is valid for time horizons $0\le T \le \frac{5}{12}$.
%The objective value is:
%\[ z=\frac{1080+450 T-11 T^2}{70} \]
The solution for $T=\frac{1}{3}$ is displayed in Figure \ref{fig.1}.  
\begin{figure}[H] 
\centering 
\caption[Example \ref{sec.twothreedim}. Solution for $T=1/3.$]{Example \ref{sec.twothreedim}. Solution for $\displaystyle T=\frac{1}{3}.$ Impulse scale 1:1} 
\includegraphics[width=4.6in]{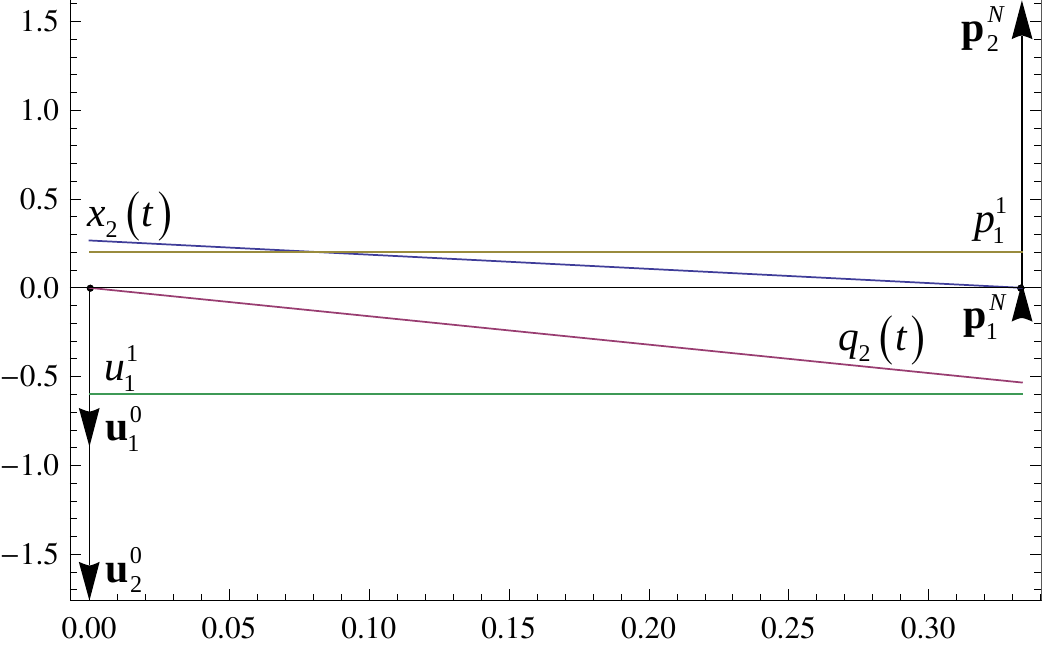} 
\label{fig.1} 
\end{figure}
In this display and in the subsequent displays we plot  all the non-zero values of the impulses and the control rates, and all the non-zero values of the states (slacks), as follows:  We plot $-\bu^0,-\bu^N,\bp^0,\bp^N$ at the left and right ends of the time horizon, and $-u(t),x(t),p(T-t),-q(T-t)$ for $0<t<T$.  We plot $x,P$ as positive, $-q,-U$ as negative values.
We scale for all the impulses different from rates and states, with the values on the vertical axis measuring  rates and states.  Recall that dual variables run in reversed time.  

 At time horizon $T = \frac{5}{12}$ the value of $\bp_1^N$ shrinks to $0$, and for larger time horizon a second interval has to be added at $T$.
For time horizons  $5/12 \le T \le 1$ the solution consists of two intervals,

%\[
%\K_0=\{2\},\,\J_0=\emptyset,\;\K_1=\{2\},\;\J_1=\{2\},\;\K_2=\{1\},\;\J_2=\{2\},\; \K_3=\{ 2\},\, \J_3= \{1\}.
%\]
{\raggedright
$ \displaystyle
\bu^0=\left[ \begin{array}{c}\displaystyle \frac{46}{49} - \frac{4}{49} T \\[0.3cm] \displaystyle \frac{81}{49} + \frac{10}{49} T \end{array} \right], \;
x^0 = \left[ \begin{array}{c}\displaystyle 0 \\[0.3cm] \displaystyle \frac{4}{7} - \frac{4}{7}T \end{array} \right]\;  \bp^0 =\left[ \begin{array}{c}\displaystyle 0 \\[0.3cm] \displaystyle 0 \end{array} \right] \; \bq^0 = \left[ \begin{array}{c}\displaystyle 0 \\[0.3cm] \displaystyle 0 \end{array} \right]$\\[0.1cm]
$ \displaystyle 
u^1 = \left[ \begin{array}{c}\displaystyle \frac{3}{5} \\[0.3cm] \displaystyle 0 \end{array} \right]\; \dx^1 =\left[ \begin{array}{c}\displaystyle 0 \\[0.3cm] \displaystyle - \frac{4}{5} \end{array} \right] \;  p^1 = \left[ \begin{array}{c}\displaystyle \frac{1}{5} \\[0.3cm] \displaystyle 0 \end{array} \right]\; \dq^1 = \left[ \begin{array}{c}\displaystyle 0 \\[0.3cm] \displaystyle - \frac{8}{5} \end{array} \right] \tau_1 = \frac{5}{7} - \frac{5}{7}T$\\[0.1cm]
$ \displaystyle 
u^2 = \left[ \begin{array}{c}\displaystyle \frac{1}{3} \\[0.3cm] \displaystyle 0 \end{array} \right]\; \dx^2 =\left[ \begin{array}{c}\displaystyle \frac{4}{3} \\[0.3cm] \displaystyle 0 \end{array} \right] \;  p^2 = \left[ \begin{array}{c}\displaystyle 0 \\[0.3cm] \displaystyle \frac{1}{3} \end{array} \right]\; \dq^2 = \left[ \begin{array}{c}\displaystyle 0 \\[0.3cm] \displaystyle - \frac{2}{3} \end{array} \right] \tau_2 = \frac{12}{7} T - \frac{5}{7}$\\[0.1cm]
$\displaystyle
\bu^N = \left[ \begin{array}{c}\displaystyle 0 \\[0.3cm] \displaystyle 0 \end{array} \right]\; \bx^N = \left[ \begin{array}{c}\displaystyle \frac{16}{7}T - \frac{20}{21} \\[0.3cm] \displaystyle 0 \end{array} \right]\; \bp^N = \left[ \begin{array}{c}\displaystyle 0 \\[0.3cm] \displaystyle \frac{5}{3} \end{array} \right],\; q^N = \left[ \begin{array}{c}\displaystyle 0 \\[0.3cm] \displaystyle \frac{2}{3} \end{array} \right]$.\\[0.3cm]}
\noindent The solution for $T=\frac{3}{4}$ is displayed in Figure \ref{fig.1.2}.
\begin{figure}[h!] 
\centering 
\caption[Example \ref{sec.twothreedim}. Solution for $ T=3/4.$]{Example \ref{sec.twothreedim}. Solution for $\displaystyle T=\frac{3}{4}.$ Impulse scale 1:3} 
\includegraphics[width=4.6in]{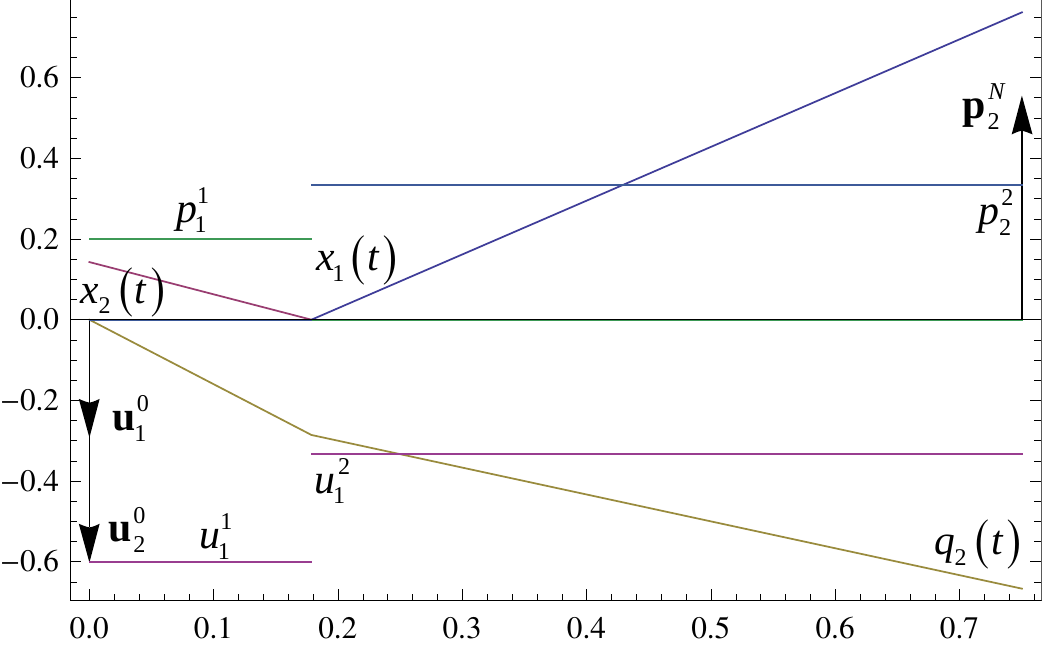} 
\label{fig.1.2} 
\end{figure}

At time horizon $T = 1$ the interval $\tau_1$ and the value of $x^0_2$ shrink to $0$.  At this time horizon $\gamma_1+ c_1 T = \gamma_2 = 6$ so the Uniqueness Condition fails, and the boundary values are not unique.  For time horizon $T=1$ and for any $0 \le \theta \le 1$ solution consists of one interval.

{\raggedright
$ \displaystyle
\bu^0=\left[ \begin{array}{c}\displaystyle \frac{6}{7} - \frac{6}{7} \theta \\[0.3cm] \displaystyle \frac{13}{7}  + \frac{9}{14} \theta \end{array} \right], \;
x^0 = \left[ \begin{array}{c}\displaystyle 3\theta \\[0.2cm] \displaystyle 0 \end{array} \right]\;  \bp^0 =\left[ \begin{array}{c}\displaystyle 0 \\[0.2cm] \displaystyle 0 \end{array} \right] \; \bq^0 = \left[ \begin{array}{c}\displaystyle 0 \\[0.2cm] \displaystyle 0 \end{array} \right]$\\[0.1cm]
$ \displaystyle 
u^1 = \left[ \begin{array}{c}\displaystyle \frac{1}{3} \\[0.2cm] \displaystyle 0 \end{array} \right]\; \dx^1 =\left[ \begin{array}{c}\displaystyle \frac{4}{3} \\[0.2cm] \displaystyle 0 \end{array} \right] \;  p^1 = \left[ \begin{array}{c}\displaystyle 0 \\[0.2cm] \displaystyle \frac{1}{3} \end{array} \right]\; \dq^1 = \left[ \begin{array}{c}\displaystyle 0 \\[0.2cm] \displaystyle - \frac{2}{3} \end{array} \right] $\\[0.1cm]
$\displaystyle
\bu^N = \left[ \begin{array}{c}\displaystyle 0 \\[0.2cm] \displaystyle 0 \end{array} \right]\; \bx^N = \left[ \begin{array}{c}\displaystyle \frac{4}{3} + 3 \theta \\[0.2cm] \displaystyle 0 \end{array} \right]\; \bp^N = \left[ \begin{array}{c}\displaystyle 0 \\[0.2cm] \displaystyle \frac{5}{3} \end{array} \right],\; q^N = \left[ \begin{array}{c}\displaystyle 0 \\[0.2cm] \displaystyle \frac{2}{3} \end{array} \right]$.\\[0.2cm]}
\noindent The objective value for this $T$ is $ z=131/6 $. The solution for time horizon $T=1$ and $\theta=1$ is displayed in Figure \ref{fig.1.3}.

For all time horizons $T\ge 1$ the optimal solution again consists of two intervals. The new interval  is inserted at $t=0$.   The solution is given by:
%\[
%\J_0=\{1\}, \K_0=\{1\}, \J_1=\{  \},  \K_1=\{  \}, \J_2=\{  \}, \K_2=\{  \}
%\J_{N+1}=\{2\}, \K_{N+1}=\{1\}.
%\]
{\raggedright
$ \displaystyle
\bu^0=\left[ \begin{array}{c}\displaystyle 0 \\[0.2cm] \displaystyle \frac{5}{2} \end{array} \right], \;
x^0 = \left[ \begin{array}{c}\displaystyle 3 \\[0.2cm] \displaystyle 0 \end{array} \right]\;  \bp^0 =\left[ \begin{array}{c}\displaystyle 0 \\[0.2cm] \displaystyle 0 \end{array} \right] \; \bq^0 = \left[ \begin{array}{c}\displaystyle \frac{1}{2} T - \frac{1}{2} \\[0.2cm] \displaystyle 0 \end{array} \right]$\\[0.1cm]
$ \displaystyle 
u^1 = \left[ \begin{array}{c}\displaystyle 0 \\[0.2cm] \displaystyle \frac{1}{4} \end{array} \right]\; \dx^1 =\left[ \begin{array}{c}\displaystyle \frac{5}{2} \\[0.2cm] \displaystyle 0 \end{array} \right] \;  p^1 = \left[ \begin{array}{c}\displaystyle 0 \\[0.2cm] \displaystyle \frac{1}{2} \end{array} \right]\; \dq^1 = \left[ \begin{array}{c}\displaystyle \frac{1}{2} \\[0.2cm] \displaystyle 0 \end{array} \right] \tau_1 = T - 1 $\\[0.1cm]
$ \displaystyle 
u^2 = \left[ \begin{array}{c}\displaystyle \frac{1}{3} \\[0.2cm] \displaystyle 0 \end{array} \right]\; \dx^2 =\left[ \begin{array}{c}\displaystyle \frac{4}{3} \\[0.2cm] \displaystyle 0 \end{array} \right] \;  p^2 = \left[ \begin{array}{c}\displaystyle 0 \\[0.2cm] \displaystyle \frac{1}{3} \end{array} \right]\; \dq^2 = \left[ \begin{array}{c}\displaystyle 0 \\[0.2cm] \displaystyle - \frac{2}{3} \end{array} \right] \tau_2 = 1$\\[0.1cm]
$\displaystyle
\bu^N = \left[ \begin{array}{c}\displaystyle 0 \\[0.2cm] \displaystyle 0 \end{array} \right]\; \bx^N = \left[ \begin{array}{c}\displaystyle \frac{11}{6} + \frac{5}{2} T \\[0.2cm] \displaystyle 0 \end{array} \right]\; \bp^N = \left[ \begin{array}{c}\displaystyle 0 \\[0.2cm] \displaystyle \frac{5}{3} \end{array} \right],\; q^N = \left[ \begin{array}{c}\displaystyle 0 \\[0.2cm] \displaystyle \frac{2}{3} \end{array} \right]$.\\[0.2cm]}
\begin{figure}[h!t] 
\centering 
\caption[Example \ref{sec.twothreedim}. Solution for $T=1.$]{Example \ref{sec.twothreedim}. Solution for $T=1$ and $\theta=1$. Impulse scale 1:3} 
\includegraphics[width=4.6in]{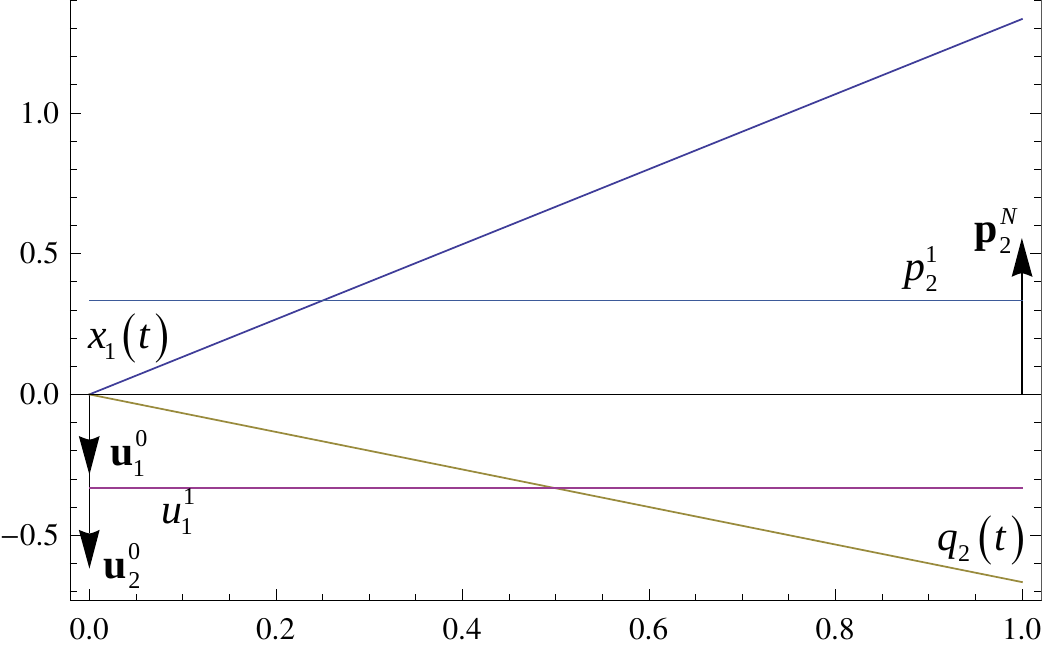} 
\label{fig.1.3} 
\end{figure}
%Objective value: 
%\[ z=\frac{181 + 78 T + 3 T^2}{12} \]  
%In this solution we have:
%This solution is valid for all $T\ge 1$.
\noindent  The solution for time horizon $T=\frac{3}{2}$ is displayed in Figure \ref{fig.1.4}.  

\begin{figure}[h!t] 
\centering 
\caption[Example \ref{sec.twothreedim}. Solution for $T=3/2.$]{Example \ref{sec.twothreedim}. Solution for $\displaystyle T=\frac{3}{2}.$ Impulse scale 1:2} 
\includegraphics[width=4.6in]{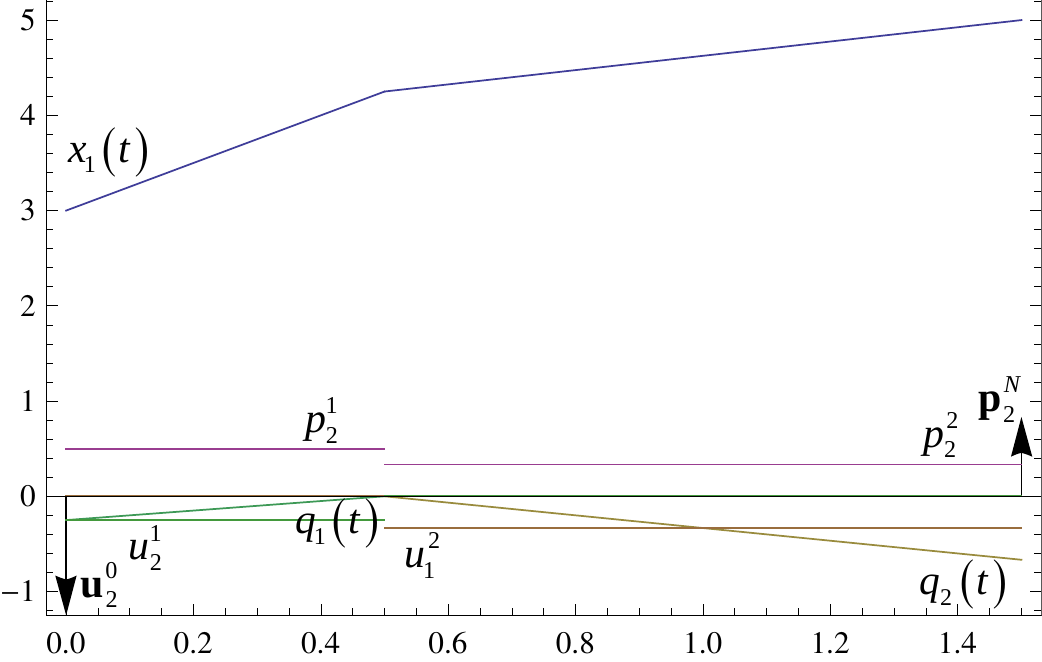} 
\label{fig.1.4} 
\end{figure}

\subsection{SCLP examples without strong duality}
\label{sec.SCLP}
These examples for which strong duality between SCLP and SCLP$^*$ fails, cannot be solved by the simplex type algorithm described in \cite{weiss:08}.
%
%%%%%%%%%%%%%%%%%%%%%%%%%%%%%%%%
\subsubsection*{Case 1: SCLP is feasible and bounded, but SCLP$^*$ is infeasible}
Consider an SCLP problem with the following data:
\[
G = 1, F = H = b = d = \emptyset, \alpha \ge 0,  a > 0, \gamma > 0, c > 0 
\]
SCLP problem with this data is feasible and bounded, but SCLP$^*$ problem is not feasible, because  the constraint $\int_0^t p(t) \ge \gamma$ cannot be satisfied with $\gamma > 0$. 
%Such pair of problems could not be solved by the simplex-type algorithm of Weiss \cite{weiss:08}, because assumption 1(i) in \cite{weiss:08} is violated. 
However, M-CLP/M-CLP$^*$ extensions of these problems are feasible and have optimal solutions described as Case 1 in Section \ref{sec.onedim}.  
We have:
\[
\sup V(SCLP) = V(\mbox{M-CLP}) = \gamma (\alpha + aT) + c T \left(\alpha + \frac{a T}{2}\right)
\]
Here SCLP has no optimal solution if $\alpha>0$, but does have the optimal solution $u(t)=a$ if $\alpha=0$.
%
%%%%%%%%%%%%%%%%%%%%%%%%%%%%%%%%%
\subsubsection*{Case 2: SCLP/SCLP$^*$ are feasible, but SCLP has no optimal solution}
Consider an SCLP problem with the following data:
\[
G = 1, F = H = b = d = \emptyset, \alpha > 0,  a > 0, \gamma < 0, c > 0 
\]
Both SCLP/SCLP$^*$  are feasible for all $T \ge 0$. Moreover, for $T < - \frac{\gamma}{c}$ both SCLP and SCLP$^*$  have optimal solutions: $u(t) = 0, x(t) = \alpha + at, \,p(t) = 0, q(t) = \gamma + ct$ with objective value $0$.

For $T> - \frac{\gamma}{c}$ the
M-CLP/M-CLP$^*$ extensions of this problem are described as Case 3 in Section \ref{sec.onedim}.
M-CLP$^*$ has an absolutely continuous solution $P(t)$, and so SCLP$^*$ is solved by $p(t)=\frac{dP(t)}{dt}$. 
However,  for $T > - \frac{\gamma}{c}$ SCLP does not have an optimal solution.   The  objective values satisfy:
\[
\sup V(SCLP) = V(\mbox{M-CLP}) = V(\mbox{M-CLP}^*)= V(SCLP^*) = \alpha (\gamma + cT) + \frac{a c T^2}{2} + a \gamma \left( T + \frac{\gamma}{2c} \right)
\]

\subsubsection*{Case 3: SCLP/SCLP$^*$ are feasible and bounded, but both have no optimal solutions}
Consider an SCLP problem with following data:
\[
G = \left[ \begin{array}{cc} 1 & 0 \\ 0 & -1 \end{array}\right], F = H = b = d = \emptyset, \alpha = \left[ \begin{array}{c} 1  \\  3 \end{array}\right],  a = \left[ \begin{array}{c} 5  \\  -1 \end{array}\right], \gamma = \left[ \begin{array}{c} -2  \\  -1 \end{array}\right], c = \left[ \begin{array}{c} 1  \\  -6 \end{array}\right] 
\] 

For $T \le 2$ both SCLP and SCLP$^*$ have the following optimal solutions:
\[
u(t) = \left[ \begin{array}{c} 0  \\  0 \end{array}\right], x(t) = \left[ \begin{array}{c} 1 + 5t \\  3 - t \end{array}\right], p(t) = \left[ \begin{array}{c} 0  \\  0 \end{array}\right], q(t) = \left[ \begin{array}{c} 2 - t \\  1 + 6t \end{array}\right]
\]
with objective value $0$. 

For $2 < T \le 3$ SCLP has no optimal solution, but SCLP$^*$ has an optimal solution: $p_1(t) = 0, q_1(t) = 2 - t$ for $t \le 2$, $p_1(t) = 1, q_1(t) = 0$ for $2 < t \le T$, $p_2(t) = 0, q_2(t) = 1 + 6t$ for all $t \le T$. The corresponding objective values satisfy:
\[
\sup V(SCLP) = V(SCLP^*) = 8 - 9T + 2.5 T^2.
\]

For $T > 3$ SCLP/SCLP$^*$ are feasible but both have no optimal solutions. The M-CLP/M-CLP$^*$ extensions have optimal solutions (note that dual run in reversed time):\\
\scalebox{0.96}{$
\begin{array}{lllll} U_1(t) =  1 + 5t,& x_1(t) = 0, & P_1(t) = t - 2,& q_1(t) = 0 & t \le T - 2 \\ 
                     U_1(t) = 5T - 9, & x_1(t) = 5(t-T+2), & P_1(t) = 0, & q_1(t) = 2 - t & t > T - 2 \\
                     U_2(t) = 0, & x_2(t) = 3 - t, & P_2(t) = 6T - 17,& q_2(t) = 6(t-T+3) &  t \le 3 \\
                     U_2(t) = t - 3, & x_2(t) = 0,  & P_2(t) = 1 + 6t, & q_2(t) = 0 &  t > 3
\end{array}                     
$}

The  objective values satisfy:
\[
\sup V(SCLP) = V(\mbox{M-CLP}) = V(\mbox{M-CLP}^*) = \inf V(SCLP^*) = - 16 + 8T - 0.5 T^2
\]


\begin{thebibliography}{20}

\bibitem{anderson:81} {\sc Anderson, E.~J. }
{\em A new continuous model for job-shop scheduling.}
International J. Systems Science, {\bf 12}, pp. ~1469--1475, 1981.

\bibitem{anderson-nash:87}
{\sc Anderson, E.~J., Nash, P. }
{\em Linear Programming in Infinite Dimensional Spaces}.
Wiley-Interscience, Chichester, 1987.

\bibitem{anderson-philpott:89}
{\sc Anderson, E.~J., Philpott, A.~B.} 
{\em A continuous time network simplex algorithm.}
Networks, {\bf 19}, pp. ~395--425, 1989


\bibitem{bellman:53}
{\sc Bellman, R.} 
{\em Bottleneck problems and dynamic programming.}
Proc. National Academy of Science {\bf 39}, pp. ~947--951, 1953.

\bibitem{dantzig:51}
{\sc Dantzig, G.}
{\em Application of the simplex method to a transportation problem.}
in T. Koopmans, ed., ‘Activity Analysis of Production and Allocation’
John Wiley and Sons, New York, pp. ~359--373, pp. ~330--335, 1951. 

\bibitem{grinold:70}
{\sc  Grinold, R.~C.} 
{\em Symmetric duality for continuous linear programs.}
SIAM J. Applied Mathematics, {\bf 18}, pp. ~32--51, 1970.

\bibitem{levinson:66}
{\sc Levinson, N.}
{\em A class of continuous linear programming problems.}
J. Mathematical Analysis Applications, {\bf 16}, pp. ~73--83, 1966.

\bibitem{pullan:93}
{\sc Pullan, M.~C.} 
{\em An algorithm for a class of continuous linear programs.}
SIAM J. Control and Optimization, {\bf 31}, pp. ~1558--1577, 1993.

\bibitem{pullan:95}
{\sc Pullan, M.~C.}
{\em Forms of optimal solutions for separated continuous linear programs.}
SIAM J. Control and Optimization, {\bf 33}, pp. ~1952--1977, 1995.

\bibitem{pullan:96}
{\sc Pullan, M.~C.}
{\em A duality theory for separated continuous linear programs.}
SIAM J. Control and Optimization, {\bf 34}, pp. ~931--965, 1996.

\bibitem{pullan:97}
{\sc Pullan, M.~C.}
{\em Existence and duality theory for separated continuous linear programs.} 
Math. Model. Syst., {\bf 3}, pp. ~219--245, 1997

\bibitem{pullan:00}
{\sc Pullan, M.~C.}
{\em  Convergence of a general class of algorithms for separated continuous linear programs.}
SIAM J. Control and Optimization, {\bf 10}, pp. ~722--731, 2000.

\bibitem{shindin-weiss:13}
{\sc Shindin, E., Weiss G.}
{\em Symmetric Strong Duality for a Class of Continuous Linear Programs with Constant Coefficients.}
SIAM J. Optimization, {\bf 24}, pp. ~1102-–1121, 2014.

\bibitem{weiss:08}
{\sc Weiss, G.}
{\em A simplex based algorithm to solve separated continuous linear programs.}
Mathematical Programming Series A, pp. ~151--198, 2008.
\end{thebibliography}
\end{document}